\documentclass[12pt,reqno]{amsart}
\advance \topmargin by -\headheight
\advance \topmargin by -\headsep
\evensidemargin \oddsidemargin
\usepackage{comment}
\usepackage{amsmath}
\usepackage{amsfonts}
\usepackage{stmaryrd}
\usepackage{amssymb}
\usepackage{amsthm}
\usepackage{csquotes}
\usepackage{color}
\usepackage{mathrsfs}
\usepackage{dsfont}
\usepackage{bm}
\usepackage{bbm}
\usepackage{cite}
\usepackage{soul}
\usepackage{enumerate}

\numberwithin{equation}{section}

\newcommand{\abs}[1]{\lvert#1\rvert}
\DeclareMathOperator{\Absset}{Abs}  
\DeclareMathOperator{\Logset}{Log}  
\DeclareMathOperator{\LogAbsset}{LogAbs}

\newtheorem{theorem}{Theorem}[section]
\newtheorem{fact}{Fact}[section]
\newtheorem{lemma}[theorem]{Lemma}

\usepackage{mathtools}

\title{A symmetric multivariate Elekes--R\'{o}nyai theorem}

\author{Yewen Sun}
\address{Department of Mathematics, The Ohio State University, Columbus, OH, 43210, USA}
\email{}

\subjclass[2020]{}

\begin{document}

\begin{abstract}We consider a polynomial $P\in \mathbb{R}[x_{1},\ldots, x_{d}]$ of degree $\delta$ that depends non-trivially on each of $x_1,\ldots,x_d$ with $d\ge 2$. For any integer $t$ with $1\le t\le d-1$, any natural number $n\in\mathbb N$, and any finite set $A\subset\mathbb R$ of size $n$, write
$$P(A,\dots,A):=\{P(a_1,\dots,a_d): a_1,\dots,a_d\in A\}.$$
Our first result shows that $$|P(A,\dots,A)|\gg_{\delta} n^{\frac{3}{2}-\frac{1}{2^{t+1}}}$$ unless
$P(x_1,\dots,x_d)=f\big(u_1(x_1)+\cdots+u_d(x_d)\big)$ or
$P(x_1,\dots,x_d)=f\big(u_1(x_1)\cdots u_d(x_d)\big)$,
for some nonconstant univariate real polynomials $f(x),u_1(x),\ldots,u_d(x)$, and there exists an index subset $I\subseteq[d]$ with $|I|=d-\lfloor \frac{t-1}{2}\rfloor$ such that for any $i,j\in I$ we have $u_i(x)=\lambda_{ij}u_j(x)$ (in the additive case) or $|u_i(x)|=|u_j(x)|^{\kappa_{ij}}$ (in the multiplicative case) for some constants $\lambda_{ij}\in\mathbb R^{\times}$ and $\kappa_{ij}\in\mathbb Q^{+}$. This result generalizes the symmetric Elekes--R\'onyai theorem proved by Jing, Roy, and Tran. 
Our second result is a generalized Erd\H{o}s--Szemer\'edi theorem for two polynomials in higher dimensions, giving a lower bound for $\max\{|P(A,\dots,A)|,|Q(A,\dots,A)|\}$, and generalizing another theorem by Jing, Roy, and Tran. A key ingredient in our proofs is a variation of a theorem by Elekes, Nathanson, and Ruzsa.
\end{abstract}

\maketitle

\section{Introduction}

 Let $f\in \mathbb{R}[x,y]\setminus(\mathbb{R}[x]\cup \mathbb{R}[y])$ be a polynomial of degree $\delta$. Assume $A,B\subset \mathbb{R}$ are two finite sets, each of size $n$, and define $$f(A,B):=\{f(a,b):a\in A, b\in B\}.$$ Elekes and  R\'{o}nyai ~\cite{elekes2000combinatorial} showed that $|f(A, B)| \gg_{\delta} n^{1+\varepsilon}$ for some $\varepsilon > 0$, unless $$f(x, y) = a(b(x) + c(y)) \text{~or~} f(x, y) = a(b(x)c(y)),$$ where $a(x), b(x), c(x)$ are univariate polynomials over $\mathbb{R}$. Many improvements and generalizations have been made. Recent efforts are mostly based on a generalization proved by Elekes--Szab\'{o} ~\cite{elekes2012find}, followed by a series of generalizations and quantitative improvements ~\cite{raz2016polynomials,raz2016polynomials2,raz2018elekes, Bays2018ProjectiveGA, raz2020expanding, solymosi2024improved}. The first significant improvement $\varepsilon=\frac{1}{3}$ was obtained by Raz--Sharir--Solymosi~\cite{raz2016polynomials2}, and the best known bound $\varepsilon =\frac{1}{2}$ was due to Solymosi--Zahl ~\cite{solymosi2024improved}.

  \par The Elekes--R\'{o}nyai theorem has a number of important applications, and we refer the interested reader to the excellent survey by de Zeeuw ~\cite{de2018survey}. Here we only emphasize one special example from Elekes, Nathanson, and Ruzsa ~\cite{elekes2000convexity}. They proved that $$|f(A,A)|\gg |A|^{5/4}$$ when $f(x,y)=x+y^2$.
Note that $f(x,y)=x+y^2$ is of the special additive form $a(b(x)+c(y))$ (with $a=\mathrm{id}$, $b(x)=x$, $c(y)=y^2$),
so the usual Elekes--R\'onyai theorem does not yield a superlinear lower bound in this case.
This motivates the search for a \emph{symmetric} variant that can still force expansion when $A=B$.
 Therefore, a natural conjecture arises, proposed by de Zeeuw ~\cite{de2018survey}, suggesting that there should be a powerful variation of the Elekes--R\'{o}nyai theorem applicable to functions such as $f(x,y)=x+y^2$, particularly in the case where $A=B$. In our paper, whenever we have a condition like $A=B$, we refer to it as the symmetric case.
  
\par This conjecture was solved by Jing, Roy, and Tran ~\cite{jing2022semialgebraic}. They proved a symmetric version of the Elekes--R\'{o}nyai theorem in two dimensions. To introduce their result better, we first define two equivalence relations $\equiv_{a}$ and $\equiv_{m}$ over polynomials $\mathbb{R}[x]$. We write $$p(x) \equiv_{a} q(x)$$ if there exists a constant $\lambda\in\mathbb R^\times $ such that $p(x)=\lambda q(x)$  for all $x\in \mathbb{R}$. We also write $$p(x) \equiv_{m} q(x)$$ if there exists a constant $\kappa\in\mathbb{Q}^{+} $ such that $|p(x)|=|q(x)|^\kappa$ for all $x\in \mathbb{R}$.

They showed that for any finite set $A \subset \mathbb{R}$, we have $|f(A, A)| \gg_{\delta} |A|^{5/4}$ unless $f$ is of the form $$f(x,y)=a(b(x)+c(y)),$$ where $a(x),b(x),c(x)\in\mathbb{R}[x]$ and $b(x)\equiv_a c(x)$, or $$f(x,y)=a(b(x)c(y)),$$ where $a(x),b(x),c(x)\in\mathbb{R}[x]$ and $b(x)\equiv_{m} c(x)$. For example, when $f(x,y)=x+y^2$, it is neither of the two forms above, so their result implies that $f(A,A)\gg |A|^{5/4}$.

 The Elekes--R\'{o}nyai theorem in high dimensions was given by 
Raz--Shem-Tov ~\cite{raz2020expanding}. Let $P \in \mathbb{R}[x_1, \ldots, x_d]$ for some $d \geq 3$ with degree $\delta$, and assume that $P$ depends non-trivially on each of $x_1, \ldots, x_d$. They proved that for finite sets $A_1, \ldots, A_d \subset \mathbb{R}$, each of size $n$,
$$
|P(A_1, \ldots ,A_d)|\gg_{\delta}n^{3 / 2},
$$
unless $P$ is of the form
$$
\begin{aligned}
& P(x_1, \ldots, x_d)=f(u_1(x_1)+\cdots+u_d(x_d)) \quad\text { or } \\
& P(x_1, \ldots, x_d)=f(u_1(x_1) \cdot \ldots \cdot u_d(x_d))
\end{aligned}
$$
for some univariate polynomials $f(x), u_1(x), \ldots, u_d(x)\in\mathbb{R}[x]$. 

It is natural to ask whether we can establish a symmetric version of the theorem above, and this is one of our main contributions in this paper. It also serves as a multivariate generalization of the results in ~\cite{jing2022semialgebraic}. We adopt a different strategy from ~\cite{jing2022semialgebraic} to reprove their result, and our approach can be generalized to higher dimensions. To state the theorem, we use the notation $[d]=\{1,\ldots,d\}$.

\begin{theorem}\label{thm:main1} 

Let $P(x_{1},\ldots,x_{d})$ be a polynomial in $ \mathbb{R}[x_{1},\ldots,x_{d}]$  for some $d\geq 2.$ Assume $P$ has degree $ \delta $ and $P$ depends non-trivially on each of $x_1,...,x_d$.  For any integer $t$ with $1\leq t\leq d-1$, $n \in \mathbb{N}$ and finite $A\subset \mathbb{R}$ with $|A|=n$, we have
$$
| P(A,\ldots, A)|\gg_{\delta} n^{\frac{3}{2}-\frac{1}{2^{t+1}}},
$$
unless one of the following holds:

(i) $P(x_{1},\ldots,x_{d})=f( u_{1}(x_{1})+\cdots +u_{d}(x_{d}))$ for some nonconstant univariate real polynomials  $f(x)$,$u_1(x),\ldots,u_{d}(x)$. 
Moreover, there exists an index subset $I\subseteq [d]$ with $|I|=d-\lfloor \frac{t-1}{2} \rfloor$ such that for any $i,j\in I$,  $u_{i}(x)\equiv_{a} u_{j}(x)$.

(ii) $P(x_{1},\ldots,x_{d})=f( u_{1}(x_{1})\cdot\ldots\cdot u_{d}(x_{d}))$ for some nonconstant univariate real polynomials  $f(x)$,$u_1(x),\ldots,u_{d}(x)$. Moreover, there exists an index subset $I\subseteq [d]$ with $|I|=d-\lfloor \frac{t-1}{2} \rfloor$ such that for any $i,j\in I$, $u_{i}(x)\equiv_{m} u_{j}(x)$.

\end{theorem}

We give some examples of Theorem~\ref{thm:main1}. 
When $d=2$ and $t=1$, Theorem~\ref{thm:main1} is exactly Theorem~1.1 in ~\cite{jing2022semialgebraic}. 
If we keep $t=1$, then Theorem~\ref{thm:main1} gives the same exponent as in ~\cite{jing2022semialgebraic}. 
To obtain an exponent better than $5/4$, we need to impose an additional structure restriction on $P$, namely $|I|=d-\lfloor \frac{t-1}{2}\rfloor$. 
We conjecture that the theorem should remain true with $|I|=d$, but proving this would likely require a different methodology.

The core of our proof approach, also presented as our second theorem later, is a generalization of the Erd\H{o}s--Szemerédi theorem. Erd\H{o}s and Szemer\'{e}di proved in ~\cite{erdos1983sums} that for any finite $A\subseteq \mathbb{Z}$, we have
\[
\max\{|A+A|,|A\cdot A|\}\gg |A|^{1+\varepsilon}
\]
for some $\varepsilon >0$. In the same paper, they brought up the conjecture that $\varepsilon$ can be arbitrarily close to $1$.   The state-of-the-art result is given by ~\cite{bloom2025control}, where $\varepsilon = \frac{1}{3} + \frac{2}{951}$. Progress toward the conjecture involves incidence geometry, building on milestone results by Elekes ~\cite{elekes1997number} and Solymosi ~\cite{solymosi2009bounding} and hence applies to a more general setting where $\mathbb{Z}$ is replaced by $\mathbb{R}$. Our proof for the second theorem will also require incidence geometry.

To generalize the Erd\H{o}s--Szemerédi theorem, one possible direction is to extend the study of $A+A$ and $A\cdot A$ to a more general setting, and there have been many works in this direction ~\cite{elekes2000combinatorial,shen2012algebraic,solymosi2009incidences}. Here, we mention a work by Jing, Roy, and Tran ~\cite{jing2022semialgebraic}. They proved the following result: Let $P,Q\in \mathbb{R}[x,y]\setminus (\mathbb{R}[x]\cup \mathbb{R}[y])$ be two polynomials with degree at most $\delta $, then for all finite $A,B\subset\mathbb{R}$ with $|A|=|B|=n$, \[\max\{|P(A,B)|,|Q(A,B)|\}\gg_{\delta} n^{5/4},\] unless either $$P(x,y)=a_{1}(b_{1}(x)+c_{1}(y)), Q(x,y)=a_{2}(b_{2}(x)+c_{2}(y)),$$ where $a_{i}(x),b_{i}(x),c_{i}(x),i=1,2$ are polynomials over $\mathbb{R}$ and $b_{1}(x)\equiv_{a}b_{2}(x), c_{1}(x)\equiv_{a}c_{2}(x)$, or $$P(x,y)=a_{1}(b_{1}(x)c_{1}(y)), Q(x,y)=a_{2}(b_{2}(x)c_{2}(y)),$$ where $a_{i}(x),b_{i}(x),c_{i}(x),i=1,2$ are polynomials over $\mathbb{R}$ and $b_{1}(x)\equiv_{m}b_{2}(x), c_{1}(x)\equiv_{m}c_{2}(x)$. 
\par Another direction is to generalize the Erd\H{o}s--Szemerédi theorem to higher dimensions. A notable result in this direction was given in ~\cite{elekes2000convexity} using incidence geometry, where it was proved that $\max\{|A^k|,|kA|\}\gg |A|^{\frac{3}{2}-\frac{1}{2^k}}$ for $k\ge 2$. 

Our second theorem generalizes the above results.

\begin{theorem}\label{thm:main2} Let $P(x_{1},\ldots,x_{d}),Q(x_{1},\ldots,x_{d})$ be two polynomials in $ \mathbb{R}[x_{1},\ldots,x_{d}]$  for some $d\geq 2.$ Assume $P,Q$ have degree at most $\delta$ and $P,Q$ depend  non-trivially on each of $x_1,...,x_d$. Let $1\leq t\leq d-1$ be a positive integer. Then for all $n \in \mathbb{N}$ and subsets $A_{1},\ldots, A_{d}$ of $\mathbb{R}$ satisfying \[|A_{i}|=n \text{~for each~} i\in [d],\] we have
$$
\max \{|P(A_{1},\ldots, A_{d})|,|Q(A_{1},\ldots, A_{d})|\} \gg_{\delta} n^{\frac{3}{2}-\frac{1}{2^{t+1}}},
$$
unless one of the following holds:

(i) $P$ and $Q$ form a $t$-additive pair, i.e., \begin{align*}&P(x_{1},\ldots,x_{d})=f( u_{1}(x_{1})+\cdots+ u_{d}(x_{d}))\quad \text{and}\\&Q(x_{1},\ldots,x_{d})=g( v_{1}(x_{1})+\cdots+v_{d}(x_{d})),\end{align*} where $f(x), g(x)$,$u_{1}(x),\ldots,u_{d}(x)$,$v_{1}(x),\ldots,v_{d}(x)$ are nonconstant univariate polynomials over $\mathbb{R}$. Moreover, \(\bigl|\{\,i\in[d]:\ u_i \not\equiv_a v_i\,\}\bigr|<t.\)

(ii) $P$ and $Q$ form a $t$-multiplicative pair, i.e.,  \begin{align*}&P(x_{1},\ldots,x_{d})=f( u_{1}(x_{1})\cdot\ldots\cdot u_{d}(x_{d}))\quad \text{and}\\&Q(x_{1},\ldots,x_{d})=g( v_{1}(x_{1}) \cdot\ldots\cdot v_{d}(x_{d})),\end{align*}
where $f(x), g(x)$,$u_{1}(x),\ldots,u_{d}(x)$,$v_{1}(x),\ldots,v_{d}(x)$  are nonconstant univariate polynomials over $\mathbb{R}$.  Moreover, \(\bigl|\{\,i\in[d]:\ u_i \not\equiv_m v_i\,\}\bigr|<t.\)

\end{theorem}

Our proof strategy is inspired by a theorem of Elekes, Nathanson, and Ruzsa ~\cite{elekes2000convexity}, where they proved that 
\begin{fact}\label{fact:enr}
	For any finite $A\subset \mathbb{R}$, define \(S=\{(a,f(a)):\ a\in A\}\)
 where $f$ is a strictly convex/concave function. Then  for any finite set $T\subset\mathbb{R}^2$, we have  
    \[	|S+T|\gg \min\{|A||T|,|A|^{\frac{3}{2}}|T|^{\frac{1}{2}}\} .\]
	\end{fact}
In our paper, we will prove variations of Fact \ref{fact:enr} by replacing $S$ with different sets.

\begin{theorem}\label{thm:enr}
	Let $p,q\in \mathbb{R}[x]$ be two polynomials with degrees at most $\delta$. For any finite $A\subset \mathbb{R}$, define
	$S=\{(f( p(a)),g(q(a))):a\in A\}$
	where $f,g$ are either $\log(|x|)$ or the identity function, and when a $\log$ appears we interpret $S$ as the set of points for which the expression is defined (i.e.\ we restrict to those $a\in A$ with $p(a)\neq 0$ if $f=\log(|x|)$, and $q(a)\neq 0$ if $g=\log(|x|)$).
	If the curve
\[
C = \{(f(p(t)),\, g(q(t))) : t \in \mathbb{R}\}
\]
(where, when a $\log$ appears, we interpret $C$ as the set of points for which the expression is defined)
is not contained in an affine line,
 then  for any finite set $T\subset\mathbb{R}^2$, we have
 \[
 |S+T|\gg_{\delta } \min\{|A||T|,|A|^{\frac{3}{2}}|T|^{\frac{1}{2}}\} .
 \]
\end{theorem}

Our paper is organized as follows. Section 2 covers the notation and preliminaries. The proof of Theorem \ref{thm:enr} is presented in Section 3. In Section 4, we apply Theorem \ref{thm:enr} to give a new proof of the main results in ~\cite{jing2022semialgebraic}, illustrating how our strategy works and establishing the induction base case. Finally, our main results, Theorem \ref{thm:main1} and Theorem \ref{thm:main2}, are proved in Section 5.

\subsection*{Acknowledgement} The author would like to thank his advisor Yifan Jing for the introduction to the topic and for providing various helpful pieces of advice.

\section{Preliminaries}
\subsection{Notations} This paper will use Vinogradov notation. We write $X\gg Y$ to mean $|Y| \leq C X$ where $C>0$ is some constant. A variation is $X\gg_z Y$, meaning that $|Y| \leq C_z X$ where $C_z>0$ is some constant depending on the parameter $z$.  We use  $X\sim Y$ to denote $X\gg Y$ and $Y\gg X$, and similarly $X\sim_{z} Y$ means that $X\gg_{z} Y$ and $Y\gg_{z} X$.  Given a positive integer $N$, we use $[N]$ to denote the set $\{1, \ldots, N\}$.

\par\noindent
\textbf{Remark on implicit constants.}
Throughout the paper, implicit constants in $\ll,\gg,\sim$ are allowed to depend on the degree $\delta$ and on the ambient dimension $d$
(and on $t$ when it appears), although for brevity we often only indicate $\delta$.

We define two equivalence relations $\equiv_{a}$ and $\equiv_{m}$ for polynomials over $\mathbb{R}[x]$. We write $p(x) \equiv_{a} q(x)$ if there exists a constant $\lambda\in\mathbb R^\times $ such that $p(x)=\lambda q(x)$ for all $x\in\mathbb{R}$. Similarly, we use $p(x) \equiv_{m} q(x)$ if there exists a constant $\kappa\in\mathbb{Q}^{+} $ such that $|p(x)|=|q(x)|^\kappa$ for all $x\in\mathbb{R}$. It is not hard to verify that they are indeed equivalence relations.

\subsection{Algebraic preliminaries}
Our proofs require some tools from algebraic geometry. First, we need the resultant polynomial.

\begin{fact}[~\cite{sendra2008rational}, Theorem 4.39, Theorem 4.21]\label{fact:resultant}
A rational plane curve may be parametrized by
$x=\frac{P(t)}{R(t)}$, $y=\frac{Q(t)}{R(t)}$
with $P,Q,R\in\mathbb R[t]$.
Define
$$F(x,y):=\mathrm{res}_t(xR-P,\;yR-Q)\in\mathbb R[x,y].$$
Then $F\!\left(\frac{P(t)}{R(t)},\frac{Q(t)}{R(t)}\right)=0$ whenever $R(t)\neq 0$.
If $G(x,y)$ denotes an implicit equation of the curve (i.e.\ an irreducible polynomial defining the
Zariski closure of the image), then $G\mid F$ in $\mathbb R[x,y]$.

Moreover, writing $m:=\deg_t(xR-P)=\max\{\deg P,\deg R\}$ and
$n:=\deg_t(yR-Q)=\max\{\deg Q,\deg R\}$, we have the degree bounds
\[
\deg_x F\le n,\qquad \deg_y F\le m,
\]
and in particular $\deg F\le m+n\le 2\max\{\deg P,\deg Q,\deg R\}$.
\end{fact}

Then we have the following result for algebraic curves.

\begin{lemma}\label{lem:curve}
Let $p(t),q(t)\in\mathbb R[t]$ with $\deg p,\deg q\le \delta$ and not both constant.
Set $$C=\{(p(t),q(t)):\ t\in\mathbb R\}\subset\mathbb R^2.$$
Then there exists an irreducible polynomial $f(x,y)\in\mathbb R[x,y]$ with $\deg f\leq 2\delta$
such that $f(p(t),q(t))\equiv 0$. In particular $C$ is contained in (part of) an irreducible
algebraic curve of degree $\leq 2\delta$.
\end{lemma}

\begin{proof}
Let $m=\deg p$ and $n=\deg q$, so $m,n\le \delta$ and $\max\{m,n\}\ge 1$.
Define
\[
F(x,y):=\mathrm{Res}_{u}\bigl(x-p(u),\,y-q(u)\bigr)\in\mathbb R[x,y].
\]
Then $F\neq 0$, and by Fact~\ref{fact:resultant} we have the degree bound
\[
\deg F\le m+n\le 2\delta.
\]

For any $s\in\mathbb R$, the two polynomials in $\mathbb R[u]$
\[
p(s)-p(u)\qquad\text{and}\qquad q(s)-q(u)
\]
have a common root at $u=s$. Hence their resultant vanishes, i.e.
\[
F\bigl(p(s),q(s)\bigr)
=\mathrm{Res}_{u}\bigl(p(s)-p(u),\,q(s)-q(u)\bigr)=0.
\]
Therefore the univariate polynomial $t\mapsto F(p(t),q(t))\in\mathbb R[t]$
vanishes on all $t\in\mathbb R$, and hence
\[
F(p(t),q(t))\equiv 0 \quad \text{in }\mathbb R[t].
\]

Factor $F=\prod_{j=1}^k f_j^{e_j}$ into irreducibles in $\mathbb R[x,y]$.
Note $(x,y)=(p(t),q(t))$, in $\mathbb R[t]$:
\[
0\equiv F(p(t),q(t))=\prod_{j=1}^k f_j(p(t),q(t))^{e_j}.
\]
Since $\mathbb R[t]$ is an integral domain, there exists $j_0$ such that
$f_{j_0}(p(t),q(t))\equiv 0$. Let $f:=f_{j_0}$. Then $f$ is irreducible,
$f(p(t),q(t))\equiv 0$, and
\[
\deg f\le \deg F\le 2\delta.
\]
Thus $C\subset Z(f)(\mathbb R)$, as claimed.
\end{proof}

We also make extensive use of the classical B\'{e}zout's theorem.

\begin{fact}[~\cite{sendra2008rational}, Theorem 2.46]\label{fact:bezout}
	 Let $p$ and $q$ be two bivariate polynomials over $\mathbb{R}$, with degrees $d_p$ and $d_q$, respectively. If $p$ and $q$ vanish simultaneously at more than $d_p d_q$ points of $\mathbb{R}^2$, then $p$ and $q$ have a common non-trivial factor.
\end{fact}

\subsection{Combinatorial preliminaries}

Our proof of Theorem \ref{thm:enr} is based on incidence geometry, which requires a special version of the Szemer\'{e}di–Trotter theorem. Given a set of points $ \Pi \text { and the set of curves } \Gamma$, define the number of incidences to be 
$$\mathcal{I}(\Pi, \Gamma):=|\{(p, \ell) \in \Pi\times \Gamma: p \in \ell\}.$$
\begin{fact}[Szemer\'{e}di-Trotter theorem for curves,~\cite{pach1998number}]\label{fact:sze} Let $\Gamma$ be a set of simple curves in the plane. Suppose that each pair of curves from $\Gamma$ intersect in $\ll_{\delta} 1$ points. Let $\Pi \subset \mathbb{R}^2$ be a set of points. Suppose that for each pair of distinct points $p, p^{\prime} \in \Pi$, there are $\ll_{\delta} 1$ curves from $\Gamma$ containing both $p$ and $p^{\prime}$. Then
$$
\mathcal{I}(\Pi, \Gamma) \ll_{\delta }|\Pi|^{2 / 3}|\Gamma|^{2 / 3}+|\Pi|+|\Gamma|.
$$

\end{fact}

For the proofs of Theorems \ref{thm:main1} and \ref{thm:main2}, we require the following established Elekes--R\'{o}nyai theorem, proved by Solymosi and Zahl ~\cite{solymosi2024improved} for $d=2$, and by Raz and Shem-Tov ~\cite{raz2020expanding}  for $d>2$.

\begin{fact}[~\cite{raz2020expanding,solymosi2024improved}]\label{thm:er}
Let $d\ge 2$ and let $F\in\mathbb R[x_1,\ldots,x_d]$ have $\deg F=\delta$ and depend non-trivially on each variable.
Then for any $A_1,\ldots, A_d\subset \mathbb{R}$ with $|A_1|=\cdots=|A_d|=n $, we have either
\[ 
|F(A_1,\ldots,A_d)|\gg_\delta n^{3/2},
\]
 or $F$ is of the form
\[
F(x_1,\ldots,x_d)=f\big(u_1(x_1)+\cdots+u_d(x_d)\big)
\quad\text{or}\quad
F(x_1,\ldots,x_d)=f\big(u_1(x_1)\cdots u_d(x_d)\big),
\]
for some nonconstant univariate polynomials $f,u_1,\ldots,u_d\in\mathbb R[x]$.
\end{fact}

\section{Proof of Theorem \ref{thm:enr} }
In this section, we prove Theorem \ref{thm:enr}.

\begin{lemma}\label{lem:translation}
Let $C\subset \mathbb{A}^2_{\mathbb{R}}$ be an irreducible affine algebraic curve with $C(\mathbb{R})\neq\varnothing$.
Suppose there exists a nonzero vector $\mathbf a\in \mathbb{R}^2$ such that
\[
C+\mathbf a := \{x+\mathbf a:\ x\in C\}=C .
\]
Then $C$ is an affine line.
\end{lemma}

\begin{proof}
Since $C$ is an irreducible affine algebraic curve, there exists an irreducible polynomial
$F(x,y)\in \mathbb{R}[x,y]$ such that $C=Z(F)$.

Fix a point $z_0\in C$. The translation invariance $C+\mathbf a=C$ implies
$z_0+n\mathbf a\in C$ for all integers $n\in\mathbb{Z}$.
Let $L$ be the affine line
\[
L:=\{z_0+t\mathbf a:\ t\in \mathbb{R}\}\subset \mathbb{A}^2_{\mathbb{R}}.
\]
Consider the univariate polynomial
\[
h(t):=F(z_0+t\mathbf a)\in \mathbb{R}[t].
\]
For every $n\in\mathbb{Z}$ we have $z_0+n\mathbf a\in C=Z(F)$, hence $h(n)=0$ for infinitely many integers $n$.
Therefore $h\equiv 0$, so $F$ vanishes identically on $L$, i.e.\ $L\subset Z(F)=C$.

Let $\ell(x,y)$ be a nonzero linear polynomial defining the line $L$ (so $L=Z(\ell)$).
Since $F$ vanishes on $L$, we have $\ell\mid F$ in $\mathbb{R}[x,y]$. As $F$ is irreducible, it follows that
$F$ is a scalar multiple of $\ell$, hence $\deg F=1$ and $C=Z(F)=L$ is an affine line.
\end{proof}
Next, we will prove two variations of \ref{lem:translation}
\begin{lemma}\label{lem:mult_translation}
Let $C\subset \mathbb{A}^2_{\mathbb{R}}$ be an irreducible affine algebraic curve and assume
$C\cap (\mathbb{R}^{\times})^2\neq\varnothing$.
Suppose there exists $(\alpha,\beta)\in(\mathbb{R}^{\times})^2$ of infinite order such that
\[
(\alpha,\beta)\cdot C := \{(\alpha x,\beta y):\ (x,y)\in C\}=C .
\]
Then there exist integers $(m,n)\in\mathbb{Z}^2\setminus\{(0,0)\}$ and a constant $r\in\mathbb{R}^{\times}$ such that
\[
C\cap (\mathbb{R}^{\times})^2 \subset \{(x,y)\in(\mathbb{R}^{\times})^2:\ x^m y^n=r\}.
\]
In particular, the set $\{(\log|x|,\log|y|):\ (x,y)\in C\cap(\mathbb{R}^{\times})^2\}$ is contained in an affine line.
\end{lemma}

\begin{proof}
Since $C$ is an irreducible affine algebraic curve, there exists an irreducible polynomial
$F(x,y)\in \mathbb{R}[x,y]$ such that $C=Z(F)$.

Define $G(x,y):=F(\alpha x,\beta y)\in\mathbb{R}[x,y]$.
The map $(x,y)\mapsto(\alpha x,\beta y)$ is an automorphism of $\mathbb{A}^2_{\mathbb{R}}$ with inverse
$(x,y)\mapsto(\alpha^{-1}x,\beta^{-1}y)$.
Hence the hypothesis $(\alpha,\beta)\cdot C=C$ implies
\[
Z(F)=C=(\alpha,\beta)\cdot C=Z(G).
\]
Both $F$ and $G$ are irreducible, and two irreducible polynomials defining the same algebraic curve
must differ by a nonzero scalar. Thus there exists $c\in\mathbb{R}^{\times}$ such that
\[
F(\alpha x,\beta y)=c\,F(x,y).
\]

Write $F(x,y)=\sum_{(i,j)\in S} a_{ij}x^iy^j$ with $a_{ij}\neq 0$ and finite support $S\subset\mathbb{Z}_{\ge 0}^2$.
Comparing coefficients in the identity $F(\alpha x,\beta y)=cF(x,y)$ gives
\[
a_{ij}\alpha^i\beta^j=c\,a_{ij}\qquad\text{for all }(i,j)\in S,
\]
hence
\[
\alpha^i\beta^j=c\qquad\text{for all }(i,j)\in S.
\]
Since $C\cap(\mathbb R^\times)^2\neq\varnothing$, the irreducible $F$ is not divisible by $x$ or $y$,
hence the support $S$ contains at least two distinct pairs $(i,j)$. Thus,
\[
\alpha^{m}\beta^{n}=1,\qquad (m,n):=(i_1-i_2,\;j_1-j_2)\neq(0,0).
\]

Fix one $(i_0,j_0)\in S$. For any $(i,j)\in S$ we have $\alpha^{i-i_0}\beta^{j-j_0}=1$, so
$(i-i_0,j-j_0)$ lies in the lattice
\[
L:=\{(u,v)\in\mathbb{Z}^2:\ \alpha^u\beta^v=1\}.
\]
Because $S$ has at least two elements, $L\neq\{(0,0)\}$. Since $(\alpha,\beta)$ has infinite order in $(\mathbb R^\times)^2$, the image
\[
\{\alpha^u\beta^v:\ (u,v)\in\mathbb Z^2\}\subset \mathbb R^\times
\]
is infinite, so the kernel $L=\{(u,v)\in\mathbb Z^2:\ \alpha^u\beta^v=1\}$ cannot have rank $2$ (otherwise the image would be finite, forcing $\alpha,\beta\in\{\pm1\}$ and hence $(\alpha,\beta)$ to have finite order). Therefore $\mathrm{rank}(L)=1$, and we may choose a primitive generator $(m,n)\in\mathbb{Z}^2\setminus\{(0,0)\}$ of $L$.

Then every $(i,j)\in S$ is of the form $(i,j)=(i_0,j_0)+t(m,n)$ for some $t\in\mathbb{Z}$.

Now work on the torus $U:=(\mathbb{R}^{\times})^2$. In the Laurent ring $\mathbb{R}[x^{\pm1},y^{\pm1}]$ we can factor
\[
x^{-i_0}y^{-j_0}F(x,y)=\sum_{t\in T} b_t\,(x^my^n)^t
\]
for some finite $T\subset\mathbb{Z}$ and coefficients $b_t\in\mathbb{R}^{\times}$.
Let $t_{\min}:=\min T$ and set
\[
P(z):=\sum_{t\in T} b_t\,z^{\,t-t_{\min}}\in\mathbb{R}[z]\setminus\{0\}.
\]
Then for every $(x,y)\in U$ we have $F(x,y)=0$ if and only if $P(x^my^n)=0$
(because the prefactor $x^{i_0}y^{j_0}(x^my^n)^{t_{\min}}$ is nonzero on $U$).

Set $C^{\times}:=C\cap U$. Since $U$ is Zariski open in $\mathbb{A}^2_{\mathbb{R}}$,
the set $C^{\times}$ is a nonempty Zariski open subset of the irreducible curve $C$, hence $C^{\times}$ is irreducible.
For any $(x,y)\in C^{\times}$ the number $x^my^n$ is real and nonzero, so $P(x^my^n)=0$ implies
$x^my^n$ is a (real, nonzero) root of $P$. Let $r_1,\dots,r_s\in\mathbb{R}^{\times}$ be the distinct real roots of $P$.
Then
\[
C^{\times}\subset \bigcup_{k=1}^s \{(x,y)\in U:\ x^my^n=r_k\}.
\]
Each set $\{x^my^n=r_k\}$ is Zariski closed in $U$ (if some exponent is negative, multiply the equation by a suitable
monomial to clear denominators). Since $C^{\times}$ is irreducible and contained in a finite union of closed sets,
it must be contained in one of them. Therefore there exists $r\in\mathbb{R}^{\times}$ such that
\[
C^{\times}=C\cap(\mathbb{R}^{\times})^2 \subset \{(x,y)\in(\mathbb{R}^{\times})^2:\ x^my^n=r\}.
\]

Finally, for any $(x,y)\in C^{\times}$ we have $x^my^n=r$, hence
$m\log|x|+n\log|y|=\log|r|$,
so $\{(\log|x|,\log|y|): (x,y)\in C^{\times}\}$ lies in the affine line
$mX+nY=\log|r|$ in $\mathbb{R}^2$.
\end{proof}

\begin{lemma}\label{lem:skew_affine}
Let $D\subset \mathbb A^2_{\mathbb R}$ be an irreducible affine algebraic curve such that
$D(\mathbb R)\cap(\mathbb R^\times)^2\neq\varnothing$.
Let $l\in\mathbb R^\times$ and let $\lambda\in\mathbb R^\times$ be of infinite multiplicative order, and set
\[
A(x,y)=(x+l,\lambda y).
\]
Then $A(D)\neq D$.
More precisely, if $A(D)=D$ then necessarily $D=\{y=0\}$.
\end{lemma}

\begin{proof}
Let $F\in\mathbb R[x,y]$ be irreducible with $D=Z(F)$.
Assume $A(D)=D$ for $A(x,y)=(x+l,\lambda y)$, where $l\neq 0$ and $\lambda$ has infinite multiplicative order.

Since $A(D)=D$, the polynomial $F\circ A$ vanishes on $D$, hence $F$ divides $F\circ A$ in $\mathbb R[x,y]$.
Applying the same argument to $A^{-1}$ gives $F\circ A$ divides $F$.
Thus $F\circ A=cF$ for some $c\in\mathbb R^\times$, i.e.
\begin{equation}\label{eq:assoc2}
F(x+l,\lambda y)=c\,F(x,y).
\end{equation}

Write $d=\deg_xF$ and
\[
F(x,y)=\sum_{i=0}^d f_i(y)x^i,\qquad f_d\neq 0.
\]
Comparing $x^d$-coefficients in \eqref{eq:assoc2} yields $f_d(\lambda y)=c f_d(y)$.
Writing $f_d(y)=\sum_j b_jy^j$, we get $b_j\lambda^j=cb_j$ for all $j$.
Since $\lambda$ has infinite order, the numbers $\lambda^j$ are pairwise distinct, so $f_d(y)=ay^m$ for some
$a\in\mathbb R^\times$ and $c=\lambda^m$.

If $d\ge 1$, comparing $x^{d-1}$-coefficients in \eqref{eq:assoc2} gives
\[
f_{d-1}(\lambda y)+dl\,f_d(\lambda y)=c f_{d-1}(y),
\]
hence (using $f_d(\lambda y)=a\lambda^my^m$ and $c=\lambda^m$)
\[
f_{d-1}(\lambda y)-\lambda^m f_{d-1}(y)=-dl\,a\,\lambda^m y^m.
\]
But the left-hand side has zero $y^m$-coefficient (it is $\sum_j u_j(\lambda^j-\lambda^m)y^j$),
while the right-hand side has $y^m$-coefficient $-dl\,a\,\lambda^m\neq 0$, contradiction.
Therefore $d=0$, so $F\in\mathbb R[y]$.

Since $D(\mathbb R)\cap(\mathbb R^\times)^2\neq\varnothing$, there is a real point on $D$ with $y\neq 0$,
so $F$ has a nonzero real root. Note $F$ is irreducible in $\mathbb R[y]$, it must be linear:
$F(y)=\alpha(y-y_0)$ with $y_0\in\mathbb R^\times$.
Then $D=\{y=y_0\}$, and $A(D)=D$ forces $y_0=\lambda y_0$.
Since $\lambda$ has infinite order, $\lambda\neq 1$, hence $y_0=0$, contradiction.
Thus $A(D)\neq D$, and in fact the only invariant irreducible curve is $\{y=0\}$.
\end{proof}

Next, we extend Fact \ref{fact:enr} to curves parametrized using logarithmic functions, i.e., curves of the form  $C = \{(f(p(t)),\, g(q(t))): t \in \mathbb{R}\}$ where $f$ or $g$ is $\log(|x|)$ or the identity map.

\begin{lemma}\label{lem:sze1}
Let $p(x),q(x)\in\mathbb R[x]$ be polynomials of degree at most $\delta$.
Let $f,g$ be functions where each is either $\log(|x|)$ or the identity map.
Define
$C=\{(f(p(t)),\,g(q(t))): t\in\mathbb R\}$,
where if a $\log$ appears we interpret $C$ as the set of points for which the expression is defined.
Assume that $C$ is not contained in an affine line.
Then for any nonzero translation vector $\mathbf a\in\mathbb R^2$ we have $|(C+\mathbf a)\cap C|\le 16\delta^2$.
\end{lemma}

\begin{proof}
Let $D\subset\mathbb A^2_{\mathbb C}$ be the Zariski closure of the image of $\phi(t)=(p(t),q(t))$.
Then $D$ is an irreducible affine algebraic curve, defined over $\mathbb R$, and by Lemma~\ref{lem:curve} we have
\[
\deg D\le 2\delta.
\]
Fix $\mathbf a=(l_x,l_y)\neq(0,0)$.

\smallskip\noindent
\emph{Case 1: $f=g=\mathrm{id}$.}
Then $C\subset D(\mathbb R)$ and $C+\mathbf a\subset (D+\mathbf a)(\mathbb R)$, hence
\[
|(C+\mathbf a)\cap C|\le |(D+\mathbf a)\cap D|.
\]
If $D+\mathbf a=D$, then Lemma~\ref{lem:translation} implies that $D$ is an affine line, so $C$ is contained in an affine line,
contradicting the assumption.
Thus $D+\mathbf a\neq D$, and Fact~\ref{fact:bezout} gives
\[
|(D+\mathbf a)\cap D|\le (\deg D)^2\le (2\delta)^2=4\delta^2.
\]
Hence $|(C+\mathbf a)\cap C|\le 4\delta^2$.

\smallskip\noindent
\emph{Case 2: $f=g=\log(|\cdot|)$.}
Here
\[
C=\{(\log|p(t)|,\log|q(t)|): t\in\mathbb R,\ p(t)q(t)\neq 0\}.
\]
Since $C$ is not contained in an affine line, in particular $C\neq\varnothing$.
Set $P(t)=p(t)^2$ and $Q(t)=q(t)^2$, and define
\[
C'=\{(\log P(t),\log Q(t)):\ t\in\mathbb R,\ P(t)Q(t)\neq 0\}.
\]
Since $(\log P(t),\log Q(t))=(2\log|p(t)|,2\log|q(t)|)$, the map $(u,v)\mapsto (u/2,v/2)$ is a bijection $C'\to C$. Thus,
\[
|(C+\mathbf a)\cap C|=|(C'+2\mathbf a)\cap C'|.
\]

Write $\mathbf a=(l_x,l_y)$ and set $\alpha=e^{l_x}>0$, $\beta=e^{l_y}>0$.
Let $\psi(t)=(P(t),Q(t))$ and let $D'\subset\mathbb A^2_{\mathbb C}$ be the Zariski closure of $\psi(\mathbb A^1_{\mathbb C})$.
Then $D'$ is an irreducible affine algebraic curve defined over $\mathbb R$, and by Lemma~\ref{lem:curve} applied to $P,Q$
we have
\[
\deg D'\le 2\max\{\deg P,\deg Q\}\le 4\delta.
\]
Define the linear map $T(x,y)=(\alpha^2 x,\beta^2 y)$ and the curve $D'_T=T(D')$.

If $(u,v)\in (C'+2\mathbf a)\cap C'$, then there exist $t,s\in\mathbb R$ with $P(t)Q(t)P(s)Q(s)\neq 0$ such that
\[
u=\log P(t)=\log P(s)+2l_x
\quad\text{and}\quad
v=\log Q(t)=\log Q(s)+2l_y.
\]
Equivalently, $P(t)=\alpha^2P(s)$ and $Q(t)=\beta^2Q(s)$, i.e.\ $\psi(t)=T(\psi(s))$.
Thus
\[
(e^u,e^v)=(P(t),Q(t))\in D'\cap D'_T.
\]
Since $(u,v)\mapsto (e^u,e^v)$ is injective on $\mathbb R^2$, we get
\[
|(C'+2\mathbf a)\cap C'|\le |D'\cap D'_T|.
\]

If $D'_T\neq D'$, then Fact~\ref{fact:bezout} gives
\[
|D'\cap D'_T|\le (\deg D')^2\le (4\delta)^2=16\delta^2,
\]
and we are done.

It remains to rule out $D'_T=D'$.
Because $\mathbf a\neq 0$, we have $(\alpha^2,\beta^2)\neq(1,1)$. Moreover $\alpha^2,\beta^2>0$, hence $(\alpha^2,\beta^2)$ has infinite order.
Also $C'\neq\varnothing$ implies that there exists $t\in\mathbb R$ with $P(t)Q(t)\neq0$, so $(P(t),Q(t))\in D'\cap(\mathbb R^\times)^2$.
Hence Lemma~\ref{lem:mult_translation} applies to $C=D'$ with $(\alpha,\beta)$ replaced by $(\alpha^2,\beta^2)$, giving integers $(m,n)\neq(0,0)$ and $r\in\mathbb R^\times$ such that
\[
D'\cap(\mathbb R^\times)^2\subset\{(x,y)\in(\mathbb R^\times)^2:\ x^m y^n=r\}.
\]
On $C'$ this implies $m\log P(t)+n\log Q(t)=\log|r|$, so $C'$ (hence $C$) is contained in an affine line,
contradicting the standing assumption that $C$ is not contained in an affine line.
Therefore $D'_T\neq D'$, and we conclude
\[
|(C+\mathbf a)\cap C|=|(C'+2\mathbf a)\cap C'|\le 16\delta^2.
\]

\smallskip\noindent
\emph{Case 3: exactly one of $f,g$ is $\log(|\cdot|)$.}
By symmetry it suffices to treat $f=\mathrm{id}$ and $g=\log(|\cdot|)$.
Here
\[
C=\{(p(t),\log|q(t)|): t\in\mathbb R,\ q(t)\neq 0\}.
\]
Let $\beta=e^{l_y}>0$.

Since $C\neq\varnothing$, there exists $t$ with $q(t)\neq 0$.
Since $p$ is nonconstant, it cannot vanish on all such $t$, hence there exists $t$ with $p(t)q(t)\neq 0$.
Thus $D\cap(\mathbb R^\times)^2\neq\varnothing$.

If $(u,v)\in (C+\mathbf a)\cap C$, then there exist $t,s$ with $q(t)q(s)\neq 0$ such that
\[
u=p(t)=p(s)+l_x
\quad\text{and}\quad
v=\log|q(t)|=\log|q(s)|+l_y,
\]
so $|q(t)|=\beta|q(s)|$ and hence $q(t)=\varepsilon\beta q(s)$ for some $\varepsilon\in\{\pm 1\}$.
Define $A_\varepsilon(x,y)=(x+l_x,\varepsilon\beta y)$ and $D_\varepsilon=A_\varepsilon(D)$.
Then $(p(t),q(t))\in D\cap D_\varepsilon$, and as in the original argument we get
\[
|(C+\mathbf a)\cap C|\le \sum_{\varepsilon\in\{\pm 1\}} |D\cap D_\varepsilon|.
\]

If $D_\varepsilon\neq D$, then Fact~\ref{fact:bezout} gives
\[
|D\cap D_\varepsilon|\le (\deg D)^2\le (2\delta)^2=4\delta^2,
\]
and hence $|(C+\mathbf a)\cap C|\le 8\delta^2\le 16\delta^2$.

It remains to rule out $D_\varepsilon=D$.
Assume for contradiction that $D_\varepsilon=D$ for some $\varepsilon\in\{\pm1\}$.

If $l_y=0$, then $\beta=e^{l_y}=1$ and $A_\varepsilon^2(x,y)=(x+2l_x,y)$, so $D+(2l_x,0)=D$.
By Lemma~\ref{lem:translation}, $D$ is an affine line. Moreover, since $D$ is invariant under a nontrivial horizontal translation, this line must be horizontal, i.e.\ $D=\{y=y_0\}$ for some $y_0\in\mathbb R$ (equivalently, the defining linear polynomial has no $x$-term).
In particular, $q(t)\equiv y_0$ is constant. Since $C\neq\varnothing$, we have $y_0\neq 0$, and then
\[
C=\{(p(t),\log|y_0|):\ t\in\mathbb R\}
\]
is contained in an affine line in $\mathbb R^2$, contradicting the standing assumption that $C$ is not contained in an affine line.
Thus $l_y\neq 0$, so $\beta\neq 1$ and $\lambda:=\varepsilon\beta$ has infinite multiplicative order.
If $l_x=0$, then $A_\varepsilon(x,y)=(x,\lambda y)$.
Applying Lemma~\ref{lem:mult_translation} to $D$ with $(\alpha,\beta)=(1,\lambda)$ yields integers $(m,n)\neq(0,0)$ and $r\in\mathbb R^\times$ such that
\[
D\cap(\mathbb R^\times)^2\subset\{(x,y)\in(\mathbb R^\times)^2:\ x^m y^n=r\}.
\]
But invariance under $(1,\lambda)$ forces $\lambda^n=1$, and since $\lambda$ has infinite order we must have $n=0$.
Hence $m\neq 0$ and the constraint becomes $x^m=r$, i.e.\ $D$ is contained in a finite union of vertical lines.
Since $D$ is irreducible, it is contained in a single vertical line, so $p(t)$ is constant and $C$ is contained in an affine line,
again a contradiction.

Therefore $l_x\neq 0$ and $l_y\neq 0$.
Then Lemma~\ref{lem:skew_affine} applies (with $l=l_x$ and $\lambda=\varepsilon\beta$) and reaches a contradiction.

Hence $D_\varepsilon\neq D$ for both $\varepsilon=\pm1$, and we conclude $|(C+\mathbf a)\cap C|\le 16\delta^2$.

\smallskip
Combining the three cases completes the proof.
\end{proof}

Now we are ready to prove Theorem \ref{thm:enr}.

\begin{proof}[Proof of Theorem \ref{thm:enr}]
Let \( f,g \in \{\log(|x|), \mathrm{id}\} \), and interpret \(C\) and \(S\) on their natural domains when a \(\log\) appears.
Write
\[
C=\{(f(p(t)),g(q(t))): t\in \mathbb R \ \text{and the expression is defined}\},\]
\[
S=\{(f(p(a)),g(q(a))): a\in A \ \text{and the expression is defined}\}.
\]
Set \( \Pi = S + T \) and \( \Gamma = \{C + t : t \in T\} \). For each \(t\in T\), the translate \(C+t\) contains the set \(S+t\subset \Pi\), hence
\[
\mathcal I(\Pi,\Gamma)\ge |S||T|.
\]

Now we check the hypotheses of Fact~\ref{fact:sze}. (If needed to match the ``simple curve'' hypothesis, decompose \(C\) into \(\ll_\delta 1\) simple arcs by cutting at finitely many singular/self-intersection points, and replace \(\Gamma\) by the collection of translates of these arcs. This changes \(|\Gamma|\) and \(\mathcal I(\Pi,\Gamma)\) by at most \(\ll_\delta 1\) factors, so we keep the same notation.)

For distinct \(t,t'\in T\),
\[
(C+t)\cap(C+t') \;=\; (C\cap (C+(t'-t)))+t,
\]
so by Lemma~\ref{lem:sze1} we have \(|(C+t)\cap(C+t')|\le 16\delta^2\), i.e.\ any two curves in \(\Gamma\) intersect in \(\ll_\delta 1\) points.

Next, fix distinct points \(P,P'\in \Pi\). If \(P,P'\in C+t\) for some \(t\in T\), then \(P-t,\,P'-t\in C\) and
\[
(P-t)-(P'-t)=P-P'.
\]
Thus \(P'-t \in C\cap (C-(P-P'))\), and the map \(u=P'-t\) determines \(t=P'-u\). Therefore the number of \(t\in T\) such that \(P,P'\in C+t\) is at most
\[
|C\cap (C-(P-P'))|\le 16\delta^2
\]
by Lemma~\ref{lem:sze1} (since \(P\neq P'\) implies \(P-P'\neq 0\)).
Hence any two points of \(\Pi\) lie on \(\ll_\delta 1\) curves from \(\Gamma\).

So Fact~\ref{fact:sze} applies and gives
\begin{equation}\label{eq:inc-bound}
\mathcal{I}(\Pi, \Gamma) \ll_{\delta }|\Pi|^{2 / 3}|\Gamma|^{2 / 3}+|\Pi|+|\Gamma|.
\end{equation}
Substituting \( |\Pi| = |S+T| \) and \( |\Gamma| = |T| \), and using \(\mathcal I(\Pi,\Gamma)\ge |S||T|\), we have
\begin{equation}\label{eq:main-ineq}
|S||T| \ll_{\delta} |S+T|^{2/3}|T|^{2/3} + |S+T| + |T|.
\end{equation}
Write \(X:=|S+T|\), \(n:=|S|\), \(m:=|T|\).
Note that \(X\ge m\) since \(s_0+T\subset S+T\) for any fixed \(s_0\in S\) (if \(S=\varnothing\), the theorem is trivial).

If \(X \le c_\delta m\) (with \(c_\delta\) large enough), then \eqref{eq:main-ineq} gives
\(nm \ll_\delta m^{4/3}\), hence \(n \ll_\delta m^{1/3}\). Thus,
\[
n^{3/2}m^{1/2}\ll_\delta m \le X,
\]
so \(X\gg_\delta n^{3/2}m^{1/2}\).

Otherwise \(X>c_\delta m\), so the \(+\,|T|\) term in \eqref{eq:main-ineq} can be absorbed, giving
\[
nm \ll_{\delta} X^{2/3}m^{2/3} + X.
\]
Hence either \(nm \ll_\delta X\), which implies \(X\gg_\delta nm\), or else \(nm \ll_\delta X^{2/3}m^{2/3}\), which implies
\[
n m^{1/3}\ll_\delta X^{2/3}
\quad\Rightarrow\quad
X\gg_\delta n^{3/2}m^{1/2}.
\]
In all cases,
\[
|S+T| = X \gg_\delta \min\{\,|S||T|,\ |S|^{3/2}|T|^{1/2}\,\}.
\]

Now let \(A_0\subset A\) be the subset where the expression defining \(S\) is defined. If a \(\log\) appears, we remove at most \(\delta\) elements per polynomial, so \(|A_0|\ge |A|-2\delta\).
Moreover, for any fixed real value \(u\), the equation \(f(p(x))=u\) has at most \(2\delta\) real solutions (since it implies \(p(x)=u\) or \(p(x)=\pm e^{u}\), depending on \(f\)), so
\[
|S|\ \ge\ |f(p(A_0))|\ \ge\ \frac{|A_0|}{2\delta}.
\]
If \(|A|\le 4\delta\), then \(\min\{|A||T|,|A|^{3/2}|T|^{1/2}\}\ll_\delta |T|\le |S+T|\), so the claim is trivial.
If \(|A|>4\delta\), then \(|A_0|\ge |A|/2\), so \(|S|\ge |A|/(4\delta)\), and therefore
\[
|S+T|\gg_\delta \min\{|A||T|,|A|^{3/2}|T|^{1/2}\}.
\]
This completes the proof.
\end{proof}

\section{The two-dimensional case}
In this section, we establish the two-dimensional case of our Theorem \ref{thm:main2}, thereby giving a new proof for the main result from ~\cite{jing2022semialgebraic}. The proof will show how our strategy works and establish the base case for induction.

In the following proof, we use the symbol $|\cdot|$ to denote both the absolute value and the cardinality. When both appear in the same expression, we use the longer form $\Big|\cdot\Big|$ for the cardinality, while $|\cdot|$ continues to represent the absolute value.

\begin{lemma}\label{lem:size}
Let $p\in \mathbb{R}[x]$ be a non-constant polynomial with degree $\delta$.
For any finite nonempty $A\subset \mathbb{R}$, define
\[
\log|p(A)|:=\{\log|p(a)|:\ a\in A,\ p(a)\neq 0\}.
\]
Then $\Big|\log|p(A)|\Big|+1\sim \Big|p(A)\Big|\sim_{\delta}\Big|A\Big|$.
\end{lemma}

\begin{proof}
First, when a $\log$ appears we interpret the set where the expression is defined and write
\[
\log|p(A)|:=\{\log|p(a)|:\ a\in A,\ p(a)\neq 0\}.
\]
Set $S:=p(A)$ and $S_0:=S\setminus\{0\}$.
The map $y\mapsto \log|y|$ on $S_0$ is at most $2$-to-$1$, so
\[
\frac12\,\Big|S_0\Big|\le \Big|\log|p(A)|\Big|\le \Big|S_0\Big|\le \Big|S\Big|.
\]
Since $\Big|S_0\Big|\ge \Big|S\Big|-1$, we have
\[
\frac12\big(\Big|S\Big|-1\big)\le \Big|\log|p(A)|\Big|\le \Big|S\Big|.
\]
As $A\neq\varnothing$ implies $\Big|S\Big|\ge 1$, it follows that
$\Big|\log|p(A)|\Big|+1\sim \Big|S\Big|=\Big|p(A)\Big|$.

Next, for any map $p:\mathbb R\to\mathbb R$ we trivially have $\Big|p(A)\Big|\le \Big|A\Big|$.
For the lower bound, note that for any $c\in\mathbb R$ the equation $p(x)=c$ has at most $\delta$
solutions. Therefore,
\[
\Big|A\Big|
=\sum_{c\in p(A)} \Big|\{a\in A:\ p(a)=c\}\Big|
\le \delta\,\Big|p(A)\Big|.
\]
Thus $\Big|p(A)\Big|\ge \frac{1}{\delta}\Big|A\Big|$, and therefore
$\Big|p(A)\Big|\sim_{\delta}\Big|A\Big|$.
\end{proof}
\begin{lemma}\label{lem:sum}
Let $p_1(x), p_2(x), q_1(x), q_2(x) \in \mathbb{R}[x]$ be nonconstant polynomials,
each of degree at most $\delta$.
The following bounds hold for any finite sets $A, B \subset \mathbb{R}$ with
$|A| = |B| = n$:
\begin{enumerate}
\item[(i)] Assume $p_1(x), p_2(x), q_1(x), q_2(x)$ have no constant terms.
If $p_{1}\not\equiv_{a}q_{1}$, then
\[
|p_1(A) + p_2(B)| \cdot |q_1(A) + q_2(B)| \gg_{\delta} n^{5/2}.
\]
\item[(ii)] Assume $p_1(x), p_2(x), q_1(x), q_2(x)$ are monic.
If $p_{1}\not\equiv_{m}q_{1}$, then
\[
|p_1(A) \cdot p_2(B)| \cdot |q_1(A) \cdot q_2(B)| \gg_{\delta} n^{5/2}.
\]
\item[(iii)] For any polynomials $p_{1},p_{2},q_{1},q_{2}$,
\[
|p_1(A) + p_2(B)| \cdot |q_1(A) \cdot q_2(B)| \gg_{\delta} n^{5/2}.
\]
\end{enumerate}
\end{lemma}

\begin{proof}
Throughout, $|X|$ denotes the cardinality of a finite set $X$,
and $\abs{\cdot}$ denotes absolute value of a real number.
For a finite set $X\subset\mathbb{R}$, interpret
\[
\Absset(X):=\{\abs{x}:\ x\in X\},\qquad
\Logset(X):=\{\log x:\ x\in X\ \text{(where defined)}\},
\]
and $\LogAbsset(X):=\Logset(\Absset(X))$.

\medskip
We will use the following standard fiber bound:
if $r\in\mathbb R[x]$ is nonconstant with $\deg r\le \delta$ and $X\subset\mathbb R$ is finite, then
\[
|r(X)|\ge \frac{|X|}{\delta}
\qquad\text{and}\qquad
|\Absset(r(X))|\ge \frac{|X|}{2\delta}.
\]
Indeed, for each $y\in\mathbb R$ the equation $r(t)=y$ has at most $\delta$ real solutions,
while $\abs{r(t)}=y$ means $r(t)=\pm y$.

\medskip
\noindent (i)
Apply Theorem~\ref{thm:enr} with $f=g=\mathrm{id}$, and set
\[
\begin{aligned}
S&=\{(p_1(a),q_1(a)):\ a\in A\},\\
T&=p_2(B)\times q_2(B).
\end{aligned}
\]
Because $p_1(0)=q_1(0)=0$, the curve $\{(p_1(t),q_1(t)):\ t\in\mathbb R\}$
is contained in an affine line iff there exist $\alpha,\beta$ not both zero such that
$\alpha p_1(t)+\beta q_1(t)\equiv 0$, equivalently $q_1\equiv c\,p_1$.
This is exactly the excluded case $p_1\equiv_a q_1$ (since both have no constant term),
hence Theorem~\ref{thm:enr} applies.

By the fiber bound, $|p_2(B)|\gg_\delta n$ and $|q_2(B)|\gg_\delta n$, so $|T|\gg_\delta n^2$.
Therefore
\[
|S+T|\gg_\delta
\min\Bigl\{|A|\,|T|,\ |A|^{3/2}|T|^{1/2}\Bigr\}
\gg_\delta n^{5/2}.
\]
Finally, $S+T\subset (p_1(A)+p_2(B))\times (q_1(A)+q_2(B))$, hence
\[
|p_1(A)+p_2(B)|\cdot |q_1(A)+q_2(B)|
\ \ge\ |S+T|\ \gg_\delta n^{5/2}.
\]
\medskip
\noindent (ii)--(iii)
Both parts follow from the same application of Theorem~\ref{thm:enr} as in (i),
after the standard reductions: remove $O_\delta(1)$ zeros so that $\log|\cdot|$ is defined,
use the fiber bounds $|r(X)|\gg_\delta |X|$ and $|\LogAbsset(r(X))|\gg_\delta |X|$,
and convert products into sums via $\log|xy|=\log|x|+\log|y|$.
For completeness, we provide the full details in Appendix~\ref{app:sumproof}.
\end{proof}

Now, we are ready to give new proofs for Theorem 1.1 and Theorem 1.2 of ~\cite{jing2022semialgebraic}. Later, we will see that this proof can be generalized to $d$ variables.  First, we will prove the two-dimensional case of Theorem \ref{thm:main2}.

\begin{theorem}\label{thm:jrt1} Let $P(x, y)$ and $Q(x, y)$ be bivariate polynomials in $\mathbb{R}[x, y] \backslash(\mathbb{R}[x] \cup\mathbb{R}[y])$ with degree at most $\delta $. Then for all $n \in \mathbb{N}$ and subsets $A$ and $B$ of $\mathbb{R}$ with $|A|=|B|=n$,
$$
\max \{|P(A, B)|,|Q(A, B)|\} \gg_{\delta} n^{5 / 4}
$$
unless one of the following holds:

(i) $P(x, y)=f(u_{1}(x)+u_{2}(y))$ and $Q(x, y)=g(v_{1}(x)+v_{2}(y))$ where $f(x)$, $g(x)$,$u_{1}(x)$,$u_{2}(x)$, $v_{1}(x)$ and $v_{2}(x)$ are nonconstant univariate polynomials over $\mathbb{R}$. Moreover, $$u_{1}(x)\equiv_{a}v_{1}(x)\text{~and~}u_{2}(x)\equiv_{a}v_{2}(x).$$

(ii) $P(x, y)=f(u_{1}(x)u_{2}(y))$ and $Q(x, y)=g(v_{1}(x) v_{2}(y))$ where  $f(x)$, $g(x)$,$u_{1}(x)$,$u_{2}(x)$,$v_{1}(x)$ and $v_{2}(x)$ are nonconstant univariate polynomials over $\mathbb{R}$. Moreover, $$u_{1}(x)\equiv_{m}v_{1}(x)\text{~and~}u_{2}(x)\equiv_{m}v_{2}(x).$$
\end{theorem}

\begin{proof}
Fix finite sets $A,B\subset\mathbb R$ with $|A|=|B|=n$.
By Fact~\ref{thm:er}, for each of $P$ and $Q$ we are in one of the following situations:
either $|P(A,B)|\gg_{\delta} n^{3/2}$ (resp.\ $|Q(A,B)|\gg_{\delta} n^{3/2}$), in which case we are done as $3/2>5/4$,
or else $P$ (resp.\ $Q$) has one of the structured forms
\[
P(x,y)=f\big(u_1(x)+u_2(y)\big)\quad\text{or}\quad P(x,y)=f\big(u_1(x)u_2(y)\big),
\]
\[
Q(x,y)=g\big(v_1(x)+v_2(y)\big)\quad\text{or}\quad Q(x,y)=g\big(v_1(x)v_2(y)\big),
\]
with $f,g,u_1,u_2,v_1,v_2\in\mathbb R[x]$ nonconstant (and with degrees bounded in terms of $\delta$).
Thus we may assume $P,Q$ admit such representations, and we analyze three cases.

\medskip
\noindent\emph{Additive-additive case.}
Assume
\[
P(x,y)=f\big(u_1(x)+u_2(y)\big),\qquad Q(x,y)=g\big(v_1(x)+v_2(y)\big).
\]
By replacing $u_i(x)$ with $u_i(x)-u_i(0)$ and adjusting $f$ accordingly (and similarly for $v_i,g$),
we may assume $u_1,u_2,v_1,v_2$ have no constant term.

If $u_1\equiv_a v_1$ and $u_2\equiv_a v_2$, then (i) holds and we are done.
Otherwise, at least one of these equivalences fails. If necessary, swap the roles of $A,B$ and simultaneously
swap $(u_1,v_1)$ with $(u_2,v_2)$ so that we may assume
$u_1\not\equiv_a v_1$.
Applying Lemma~\ref{lem:sum}(i) gives
\[
|u_1(A)+u_2(B)|\cdot |v_1(A)+v_2(B)|\gg_{\delta} n^{5/2}.
\]
Now Lemma~\ref{lem:size} (applied to $f$ on the set $u_1(A)+u_2(B)$, and to $g$ on $v_1(A)+v_2(B)$) gives
\[
|P(A,B)|=|f(u_1(A)+u_2(B))|\gg_{\delta} |u_1(A)+u_2(B)|,
\]
\[
|Q(A,B)|=|g(v_1(A)+v_2(B))|\gg_{\delta} |v_1(A)+v_2(B)|.
\]
Multiplying and using the previous bound,
\[
|P(A,B)|\cdot |Q(A,B)|\gg_{\delta} n^{5/2}.
\]
Hence $\max\{|P(A,B)|,|Q(A,B)|\}\ge \sqrt{|P(A,B)|\,|Q(A,B)|}\gg_{\delta} n^{5/4}$.

\medskip
\noindent\emph{Multiplicative--multiplicative case and mixed case.}
These follow from the same argument as above after applying Lemma~\ref{lem:sum}(ii) and Lemma~\ref{lem:sum}(iii),
respectively, and then Lemma~\ref{lem:size}.
For completeness we provide full proofs in Appendix~\ref{app:jrt-cases}.
\end{proof}

\section{Proof of Theorem \ref{thm:main1} and  Theorem \ref{thm:main2}}

Now we are ready to prove Theorem \ref{thm:main2}.

\begin{proof}[Proof of Theorem \ref{thm:main2}]
We argue by contradiction. Assume that
\[
\max\{|P(A_{1},\ldots,A_{d})|,\ |Q(A_{1},\ldots,A_{d})|\}\ll_{\delta} n^{\frac32-\frac{1}{2^{t+1}}}
\]
and that neither exceptional alternative (i) nor (ii) holds.

By Fact~\ref{thm:er}, since the exponent $\frac32-\frac1{2^{t+1}}<\frac32$, the above failure forces both $P$ and $Q$ to be in one of the structured forms. The three possibilities are: both additive, both multiplicative, or one additive and one multiplicative.

First assume we are in the additive--additive case:
\[
P(x_{1},\ldots,x_{d})=f\!\Big(\sum_{i=1}^{d}u_{i}(x_{i})\Big),
\qquad
Q(x_{1},\ldots,x_{d})=g\!\Big(\sum_{i=1}^{d}v_{i}(x_{i})\Big),
\]
with $f,g,u_{1},\ldots,u_{d},v_{1},\ldots,v_{d}\in\mathbb R[x]$ nonconstant. Absorbing constants into $f$ and $g$, we may assume each $u_i$ and $v_i$ has no constant term. Since we are \emph{not} in Theorem~\ref{thm:main2}(i), we have
\[
\bigl|\{\,i\in[d]:\ u_i\not\equiv_a v_i\,\}\bigr|\ge t.
\]
Write
\[
U_d:=\sum_{i=1}^{d}u_i(A_i),\qquad V_d:=\sum_{i=1}^{d}v_i(A_i).
\]
We prove the following statement for integers $d\ge 2$ and $1\le t\le d-1$.

\smallskip
\noindent{\bf Statement $\mathcal P(d,t)$.}
If there exists $I\subseteq[d]$ with $|I|=t$ such that $u_i\not\equiv_a v_i$ for all $i\in I$, then
\[
|U_d|\cdot |V_d|\gg_{\delta} n^{3-\frac1{2^{t}}}.
\]

\smallskip
\noindent We know $\mathcal P(2,1)$ holds by Lemma~\ref{lem:sum}(i) (applied after removing constant terms).

\smallskip
\noindent We now show that $\mathcal P(d,t)$ implies both $\mathcal P(d+1,t+1)$ and $\mathcal P(d+1,t)$.

\smallskip
\noindent{\bf (1) $\mathcal P(d,t)\Rightarrow \mathcal P(d+1,t+1)$.}
Assume $\mathcal P(d,t)$. Suppose for $d+1$ variables there is $I\subseteq[d+1]$ with $|I|=t+1$ and $u_i\not\equiv_a v_i$ for all $i\in I$. Relabel indices so that $1\in I$. Set
\[
S:=\{(u_1(a),\,v_1(a)):\ a\in A_1\}\subset\mathbb R^2,
\qquad
T:=U'\times V',
\]
where
\[
U':=\sum_{i=2}^{d+1}u_i(A_i),\qquad V':=\sum_{i=2}^{d+1}v_i(A_i).
\]
Then $S+T\subseteq (u_1(A_1)+U')\times (v_1(A_1)+V')$, hence
\[
|U_{d+1}|\cdot |V_{d+1}|=|u_1(A_1)+U'|\cdot |v_1(A_1)+V'|\ \ge\ |S+T|.
\]
Because $u_1\not\equiv_a v_1$ and both have no constant term, the curve $\{(u_1(x),v_1(x)):\ x\in\mathbb R\}$ is not contained in an affine line. Thus Theorem~\ref{thm:enr} (with $f=g=\mathrm{id}$, $p=u_1$, $q=v_1$, and $A=A_1$) gives
\[
|S+T|\gg_{\delta}\min\{\,|A_1||T|,\ |A_1|^{3/2}|T|^{1/2}\,\}.
\]
By the induction hypothesis $\mathcal P(d,t)$ applied to the $d$ remaining indices (since $I\setminus\{1\}$ has size $t$),
\[
|T|=|U'|\cdot |V'|\gg_{\delta} n^{3-\frac1{2^{t}}}.
\]
In particular $|T|\ge n$ for $n$ large, so the minimum above is the second term and we obtain
\[
|S+T|\gg_{\delta} n^{3/2}\,|T|^{1/2}\gg_{\delta} n^{3/2}\cdot
\Big(n^{3-\frac1{2^{t}}}\Big)^{1/2}
=
n^{3-\frac1{2^{t+1}}},
\]
which is exactly $\mathcal P(d+1,t+1)$.

\smallskip
\noindent{\bf (2) $\mathcal P(d,t)\Rightarrow \mathcal P(d+1,t)$.}
Let $I\subseteq[d+1]$ with $|I|=t$ and $u_i\not\equiv_a v_i$ for all $i\in I$. If $d+1\notin I$, then applying $\mathcal P(d,t)$ to indices $1,\dots,d$ yields
\[
\Big|\sum_{i=1}^{d}u_i(A_i)\Big|\cdot \Big|\sum_{i=1}^{d}v_i(A_i)\Big|\gg_{\delta} n^{3-\frac1{2^{t}}}.
\]
Since adding a nonempty set cannot decrease cardinality,
\[
|U_{d+1}|=\Big|\sum_{i=1}^{d}u_i(A_i)+u_{d+1}(A_{d+1})\Big|\ge \Big|\sum_{i=1}^{d}u_i(A_i)\Big|,
\]
and similarly $|V_{d+1}|\ge \big|\sum_{i=1}^{d}v_i(A_i)\big|$, hence $\mathcal P(d+1,t)$ follows. If $d+1\in I$, the same argument applies after relabeling.

\smallskip
\noindent Therefore $\mathcal P(d,t)$ holds for all $d\ge 2$ and $1\le t\le d-1$.

\smallskip
\noindent Returning to our situation, since $\bigl|\{i:\ u_i\not\equiv_a v_i\}\bigr|\ge t$, we get
\[
|U_d|\cdot |V_d|\gg_{\delta} n^{3-\frac1{2^{t}}}.
\]
Finally, since $\deg f,\deg g\le \delta$, every value of $f$ (resp.\ $g$) has at most $\delta$ preimages, hence for any finite set $X\subset\mathbb R$ we have $|f(X)|\ge |X|/\delta$ and $|g(X)|\ge |X|/\delta$. Applying this to $X=U_d$ and $X=V_d$ gives
\[
|P(A_{1},\ldots,A_{d})|\gg_{\delta} |U_d|,
\qquad
|Q(A_{1},\ldots,A_{d})|\gg_{\delta} |V_d|.
\]
Therefore
\[
|P(A_{1},\ldots,A_{d})|\cdot |Q(A_{1},\ldots,A_{d})|
\gg_{\delta}
|U_d|\cdot |V_d|
\gg_{\delta}
n^{3-\frac1{2^{t}}},
\]
so
\[
\max\{|P(A_{1},\ldots,A_{d})|,\ |Q(A_{1},\ldots,A_{d})|\}
\ge
\big(|P|\cdot |Q|\big)^{1/2}
\gg_{\delta}
n^{\frac32-\frac1{2^{t+1}}},
\]
contradicting the assumption.

\smallskip
\noindent For the other two cases, we only indicate the necessary modifications to the additive--additive argument. Full details are deferred to Appendix~\ref{app:mm}.
 The multiplicative--multiplicative case is identical after applying $\log|\cdot|$ to both coordinates (and discarding the $O_\delta(1)$ elements where some $u_i(a)$ or $v_i(a)$ vanishes so that $\log|\cdot|$ is defined), and using $\equiv_m$ in place of $\equiv_a$. The mixed additive--multiplicative case is treated the same way with one identity coordinate and one $\log|\cdot|$ coordinate. In particular, it yields an exponent at least $\frac32-\frac1{2^{d}}$, which is stronger than $\frac32-\frac1{2^{t+1}}$ since $t+1\le d$.
\end{proof}

To prove Theorem \ref{thm:main1}, we still need several lemmas. First, we need to prove a generalization of Lemma 5.1 in ~\cite{jing2022semialgebraic}.

\begin{lemma}\label{lem:poly}
Let $d\geq 2$ and $f,g,u_{1},\ldots,u_{d},v_{1},\ldots,v_{d}\in \mathbb{R}[x]$ be nonconstant polynomials.

(i)  Assume $u_{i}$ and $v_{i}$ have no constants for all $i\in[d]$. If
$$f(u_{1}(x_{1})+\cdots+u_{d}(x_{d}))=g(v_{1}(x_{1})+\cdots+v_{d}(x_{d}))$$
for all $x_{1},\ldots,x_{d}\in\mathbb{R}$, then $u_{i}\equiv_{a} v_{i}$ for all $i\in[d]$.

(ii) Assume $u_{i}$ and $v_{i}$ are monic. If
$$f(u_{1}(x_{1})\cdot\ldots\cdot u_{d}(x_{d}))=g(v_{1}(x_{1}) \cdot\ldots\cdot v_{d}(x_{d}))$$
for all $x_{1},\ldots,x_{d}\in\mathbb{R}$, then $u_{i}\equiv_{m} v_{i}$ for all $i\in[d].$
\end{lemma}

\begin{proof}
(i) Let $U = u_1(x_1) + \cdots + u_d(x_d)$ and $V = v_1(x_1) + \cdots + v_d(x_d)$.
Taking partial derivatives with respect to $x_i$ gives
\[
f'(U)\,u_i'(x_i)=g'(V)\,v_i'(x_i)\qquad (i\in[d]).
\]
Thus
\[
\frac{f'(U)}{g'(V)}=\frac{v_i'(x_i)}{u_i'(x_i)}.
\]
The left-hand side depends on all variables, while the right-hand side depends only on $x_i$; hence each ratio
$v_i'(x)/u_i'(x)$ is a constant $\alpha_i$.
Comparing $i$ and $j$ shows $\alpha_i=\alpha_j$ for all $i,j$, so $v_i'(x)=\alpha u_i'(x)$ for a single $\alpha\in\mathbb R$.
Integrating, $v_i(x)=\alpha u_i(x)+\beta_i$. Since both have zero constant term, $\beta_i=0$.
Hence $u_i\equiv_a v_i$ for all $i$.

(ii) Let $U = \prod_{i=1}^d u_i(x_i)$ and $V=\prod_{i=1}^d v_i(x_i)$.
Differentiating with respect to $x_i$ (in the fraction field of $\mathbb R[x_1,\dots,x_d]$) gives
\[
f'(U)\,U\,\frac{u_i'(x_i)}{u_i(x_i)}=g'(V)\,V\,\frac{v_i'(x_i)}{v_i(x_i)}\qquad (i\in[d]).
\]
Fix $i\neq j$ and divide the $i$-th identity by the $j$-th identity:
\[
\frac{u_i'(x_i)/u_i(x_i)}{u_j'(x_j)/u_j(x_j)}=
\frac{v_i'(x_i)/v_i(x_i)}{v_j'(x_j)/v_j(x_j)}.
\]
Equivalently,
\[
\frac{v_i'(x_i)/v_i(x_i)}{u_i'(x_i)/u_i(x_i)}=
\frac{v_j'(x_j)/v_j(x_j)}{u_j'(x_j)/u_j(x_j)}\qquad\text{for all }x_i,x_j.
\]
Hence for each $i$ the ratio
\[
k_i:=\frac{v_i'(x)/v_i(x)}{u_i'(x)/u_i(x)}\in\mathbb R(x)
\]
is constant in $x$, and moreover $k_i=k_j$ for all $i,j$.
Thus there exists a single constant $k\in\mathbb R$ such that
\[
\frac{v_i'(x)}{v_i(x)}=k\,\frac{u_i'(x)}{u_i(x)}\qquad\text{in }\mathbb R(x)\quad(i\in[d]).
\]
Taking $x\to\infty$ shows
\[
\frac{v_i'(x)}{v_i(x)}=\frac{\deg v_i}{x}+O\!\left(\frac1{x^2}\right),
\qquad
\frac{u_i'(x)}{u_i(x)}=\frac{\deg u_i}{x}+O\!\left(\frac1{x^2}\right),
\]
hence $k=\deg(v_i)/\deg(u_i)\in\mathbb Q_{>0}$.
Write $k=p/q$ in lowest terms with $p,q\in\mathbb N$.
Then
\[
q\,\frac{v_i'}{v_i}-p\,\frac{u_i'}{u_i}=0
\quad\Longrightarrow\quad
\left(\frac{v_i^q}{u_i^p}\right)'= \frac{v_i^q}{u_i^p}\left(q\,\frac{v_i'}{v_i}-p\,\frac{u_i'}{u_i}\right)=0.
\]
Over characteristic $0$, the only rational functions with zero derivative are constants, so $v_i^q/u_i^p=C_i$ for some $C_i\in\mathbb R^\times$.
Thus
\[
v_i(x)^q = C_i\,u_i(x)^p.
\]
Since $u_i$ and $v_i$ are monic, comparing leading coefficients gives $C_i=1$, hence $v_i^q=u_i^p$.
Therefore $|u_i(x)|^p=|v_i(x)|^q$ for all $x$, i.e.\ $u_i\equiv_m v_i$ (with exponent $\kappa=q/p\in\mathbb Q^+$).
\end{proof}

We need another combinatorial lemma.

\begin{lemma}\label{lem:counting}
Let $d\geq 1$ be a positive integer and let $\equiv$ be an equivalence relation on the set $[d] := \{1,2,\dots,d\}$.
Denote by $\mathfrak{S}_d$ the group of permutations on $[d]$.
For an integer $t\in [d]$, assume that for every permutation $\sigma \in \mathfrak{S}_d$, there exists a subset
$I_{\sigma}\subseteq [d]$ with $|I_{\sigma}|=t$ satisfying
\[
i\equiv \sigma(i) \quad \text{for all } i\in I_{\sigma}.
\]
Then there exists at least one equivalence class whose size is at least
\[
\left\lceil \frac{d+t}{2} \right\rceil.
\]
\end{lemma}

\begin{proof}
Let $E_1,\dots,E_k$ be the equivalence classes, ordered so that $|E_1|\ge \cdots \ge |E_k|$.
Write $e_s:=|E_s|$ and set $e_1=:m$.
Relabel $[d]$ so that each class is a consecutive block:
\[
E_1=\{1,\dots,m\},\quad E_2=\{m+1,\dots,m+e_2\},\ \dots
\]
and $\sum_{s=1}^k e_s=d$.

Define the cyclic shift permutation
\[
\sigma(i):= i+m \pmod d,
\]
with values taken in $\{1,\dots,d\}$.
Let
\[
F_\sigma:=\{i\in[d]:\ i\equiv \sigma(i)\}.
\]
Then
\[
|F_\sigma|
=\sum_{s=1}^k |\{i\in E_s:\ \sigma(i)\in E_s\}|
=\sum_{s=1}^k |\sigma(E_s)\cap E_s|.
\]

We claim $|F_\sigma|<t$ provided $m<\left\lceil\frac{d+t}{2}\right\rceil$.
This contradicts the hypothesis (because then no subset of size $t$ can be contained in $F_\sigma$), and hence forces
$m\ge\left\lceil\frac{d+t}{2}\right\rceil$.

\smallskip\noindent
\emph{Case 1: $m\le d/2$.}
Then $2m\le d$, so shifting by $m$ does not wrap around on any block of length $\le m$.
Since each $E_s$ has size $e_s\le m$, the translate $E_s+m$ starts at least $e_s$ positions to the right of $E_s$,
hence $\sigma(E_s)\cap E_s=\varnothing$ for all $s$.
Thus $|F_\sigma|=0<t$.

\smallskip\noindent
\emph{Case 2: $m>d/2$.}
Let $m':=d-m$ (so $m'<d/2$), and note that $\sigma$ is also the shift by $-m'$.
For $s\ge 2$ we have $e_s\le \sum_{r\ge 2} e_r = d-m = m'$, hence shifting the block $E_s$ by $-m'$ moves it by at least its own length,
and again $\sigma(E_s)\cap E_s=\varnothing$ for all $s\ge 2$.

For $E_1=\{1,\dots,m\}$, the shift by $m$ wraps around, and one checks directly that
\[
\sigma(E_1)=\{m+1,\dots,d\}\cup\{1,\dots,2m-d\},
\]
so
\[
|\sigma(E_1)\cap E_1|=2m-d.
\]
Therefore $|F_\sigma|=2m-d$.

Now assume $m<\left\lceil\frac{d+t}{2}\right\rceil$. Then
\[
2m-d \le 2\left(\left\lceil\frac{d+t}{2}\right\rceil-1\right)-d < t,
\]
so again $|F_\sigma|<t$.

\smallskip
In both cases, if $m<\left\lceil\frac{d+t}{2}\right\rceil$ we get $|F_\sigma|<t$, contradicting the assumption.
Hence $m\ge \left\lceil\frac{d+t}{2}\right\rceil$, i.e.\ some equivalence class has size at least that large.
\end{proof}

Now we use Theorem \ref{thm:main2} to prove Theorem \ref{thm:main1}.
\begin{proof}
Assume for contradiction that there exist $n\in\mathbb N$ and $A\subset\mathbb R$ with $|A|=n$ such that
$$|P(A,\ldots,A)|\ll_{\delta} n^{\frac32-\frac1{2^{t+1}}},$$
and that neither alternative (i) nor (ii) of Theorem~\ref{thm:main1} holds.

For $\sigma\in\mathfrak S_d$, write
$P^{\sigma}(x_1,\ldots,x_d):=P(x_{\sigma(1)},\ldots,x_{\sigma(d)})$.
Then $P^{\sigma}(A,\ldots,A)=P(A,\ldots,A)$ for every $\sigma$.
Applying Theorem~\ref{thm:main2} to the pair $(P,P^{\sigma})$ with $A_1=\cdots=A_d=A$, we have that for every $\sigma$
the pair $(P,P^{\sigma})$ forms either an additive pair or a multiplicative pair (since the expansion conclusion fails by assumption).

We first note that $P$ cannot be simultaneously of additive and multiplicative type.
Indeed, suppose for contradiction that
\[
P=f\big(u_1(x_1)+\cdots+u_d(x_d)\big)=g\big(w_1(x_1)\cdots w_d(x_d)\big)
\]
with all $u_i,w_i$ nonconstant.
Then in the fraction field of $\mathbb R[x_1,\ldots,x_d]$, for $i\neq j$ we have
\[
\frac{\partial_{x_i}P}{\partial_{x_j}P}
=\frac{u_i'(x_i)}{u_j'(x_j)}
=\frac{w_i'(x_i)/w_i(x_i)}{w_j'(x_j)/w_j(x_j)}.
\]
Fix $j$ and vary $x_i,x_j$; separation of variables implies there exists $c_i\in\mathbb R^\times$ such that
\[
u_i'(x)=c_i\,\frac{w_i'(x)}{w_i(x)}\qquad\text{in }\mathbb R(x).
\]
Clearing denominators gives $w_i(x)\,u_i'(x)=c_i\,w_i'(x)$ in $\mathbb R[x]$, hence $w_i\mid w_i'$.
This is impossible for a nonconstant polynomial (since $\deg w_i'<\deg w_i$), a contradiction.
Therefore, exactly one of the two types works for $P$.

\smallskip
\noindent\emph{Additive case.}
Assume $P$ is of additive type. Then (taking $\sigma=\mathrm{id}$ above) we may write
$P(x_1,\ldots,x_d)=f(u_1(x_1)+\cdots+u_d(x_d))$
with $f,u_1,\ldots,u_d$ nonconstant.
By subtracting constants from each $u_i$ and absorbing the total constant into $f$,
we may assume $u_i(0)=0$ for all $i$.

When we apply Theorem~\ref{thm:main2} to $(P,P^\sigma)$, the $t$-additive pair conclusion a priori provides
some representation $P=\widetilde f(\sum_{i=1}^d \widetilde u_i(x_i))$.
After shifting constants so that $\widetilde u_i(0)=0$ for all $i$ (absorbing the total constant into $\widetilde f$),
Lemma~\ref{lem:poly}(i) applied to
\[
f\Big(\sum_{i=1}^d u_i(x_i)\Big)=\widetilde f\Big(\sum_{i=1}^d \widetilde u_i(x_i)\Big)
\]
implies $\widetilde u_i\equiv_a u_i$ for every $i$.
Hence, up to $\equiv_a$, we may assume the $t$-additive pair decomposition of $P$ uses our fixed $u_i$.

Define an equivalence relation on $[d]$ by $i\sim j$ iff $u_i\equiv_a u_j$.

Fix an arbitrary $\sigma\in\mathfrak S_d$.
Since $P$ is not of multiplicative type, the pair $(P,P^{\sigma})$ must form a $t$-additive pair.
Thus there exist nonconstant univariate polynomials $g,v_1,\ldots,v_d$ such that
$$P^{\sigma}(x_1,\ldots,x_d)=g(v_1(x_1)+\cdots+v_d(x_d)),$$
and the mismatch set
$M_{\sigma}:=\{i\in[d]: u_i\not\equiv_a v_i\}$
satisfies $|M_{\sigma}|<t$.
On the other hand,
\begin{align*}
P^{\sigma}(x_1,\ldots,x_d)
&:=P(x_{\sigma(1)},\ldots,x_{\sigma(d)})\\
&=f\big(u_1(x_{\sigma(1)})+\cdots+u_d(x_{\sigma(d)})\big)\\
&=f\big(u_{\sigma^{-1}(1)}(x_1)+\cdots+u_{\sigma^{-1}(d)}(x_d)\big).
\end{align*}
After replacing each $v_i(x)$ by $v_i(x)-v_i(0)$ and composing $g$ with a translation (so that the equality defining $P^\sigma$ is unchanged),
we may assume $v_i(0)=0$ for all $i$; this normalization does not increase the mismatch set because $u_i(0)=0$ for all $i$, and if $u_i\equiv_a v_i$ then necessarily $v_i(0)=0$ as well.
Applying Lemma~\ref{lem:poly}(i) to
$$f(u_{\sigma^{-1}(1)}(x_1)+\cdots+u_{\sigma^{-1}(d)}(x_d))
=g(v_1(x_1)+\cdots+v_d(x_d))$$
gives $u_{\sigma^{-1}(i)}\equiv_a v_i$ for all $i\in[d]$.

Therefore, for every $\sigma$ there exists a set $I_{\sigma}\subseteq[d]$ with
$|I_{\sigma}|\ge d-(t-1)=d-t+1$
such that $i\sim \sigma^{-1}(i)$ for all $i\in I_{\sigma}$.
Since $\sigma$ ranges over all permutations, so does $\sigma^{-1}$, hence we may reindex and conclude:
for every $\tau\in\mathfrak S_d$ there exists $J_{\tau}\subseteq[d]$ with $|J_{\tau}|=d-t+1$ such that
$i\sim \tau(i)$ for all $i\in J_{\tau}$.

Now apply Lemma~\ref{lem:counting} with this equivalence relation and with the parameter
$t':=d-t+1$.
We obtain an equivalence class $S\subseteq[d]$ with
$|S|\ge \left\lceil \frac{d+t'}2\right\rceil
=\left\lceil \frac{2d-t+1}{2}\right\rceil
=d-\left\lfloor \frac{t-1}{2}\right\rfloor$,
and $u_i\equiv_a u_j$ for all $i,j\in S$.
This is exactly alternative (i) in Theorem~\ref{thm:main1} (taking $I=S$).

\smallskip
\noindent\emph{Multiplicative case.}
Assume $P$ is of multiplicative type. Then (taking $\sigma=\mathrm{id}$ above) we may write
\[
P(x_1,\ldots,x_d)=f\big(u_1(x_1)\cdots u_d(x_d)\big),
\]
with $f,u_1,\ldots,u_d$ nonconstant.
By dividing each $u_i$ by its leading coefficient and absorbing the resulting nonzero constant into $f$,
we may assume that every $u_i$ is monic.

Similarly, in the multiplicative setting, any representation $P=\widetilde f(\prod_{i=1}^d \widetilde u_i(x_i))$
can be normalized so that each $\widetilde u_i$ is monic (absorbing constants into $\widetilde f$).
Then Lemma~\ref{lem:poly}(ii) implies $\widetilde u_i\equiv_m u_i$ for all $i$.
Thus we may assume the $t$-multiplicative pair decomposition of $P$ uses our fixed monic $u_i$, up to $\equiv_m$.

Define an equivalence relation on $[d]$ by $i\sim j$ iff $u_i\equiv_m u_j$.

Fix an arbitrary $\sigma\in\mathfrak S_d$.
Since $P$ is not of additive type, the pair $(P,P^\sigma)$ must form a $t$-multiplicative pair.
Thus there exist nonconstant univariate polynomials $g,v_1,\ldots,v_d$ such that
\[
P^{\sigma}(x_1,\ldots,x_d)=g\big(v_1(x_1)\cdots v_d(x_d)\big),
\]
and the mismatch set
\[
M_{\sigma}:=\{i\in[d]: u_i\not\equiv_m v_i\}
\]
satisfies $|M_\sigma|<t$.
Normalizing as above, we may assume each $v_i$ is monic: write $v_i(x)=c_i\,\widetilde v_i(x)$ with $\widetilde v_i$ monic and $c_i\in\mathbb R^\times$, and absorb $\prod_i c_i$ into the outer polynomial $g$.
This does not increase the mismatch set, because if $u_i$ is monic and $u_i\equiv_m v_i$ then the leading coefficient of $v_i$ has $|c_i|=1$, and dividing by such a constant does not affect $\equiv_m$.

On the other hand,
\begin{align*}
P^{\sigma}(x_1,\ldots,x_d)
&=P(x_{\sigma(1)},\ldots,x_{\sigma(d)})\\
&=f\big(u_1(x_{\sigma(1)})\cdots u_d(x_{\sigma(d)})\big)\\
&=f\big(u_{\sigma^{-1}(1)}(x_1)\cdots u_{\sigma^{-1}(d)}(x_d)\big).
\end{align*}
Applying Lemma~\ref{lem:poly}(ii) to the identity
\[
f\big(u_{\sigma^{-1}(1)}(x_1)\cdots u_{\sigma^{-1}(d)}(x_d)\big)
=
g\big(v_1(x_1)\cdots v_d(x_d)\big),
\]
we obtain
\[
u_{\sigma^{-1}(i)}\equiv_m v_i \qquad \text{for all } i\in[d].
\]
Hence for every $i\in[d]\setminus M_\sigma$ we have
\[
u_i\equiv_m v_i\equiv_m u_{\sigma^{-1}(i)},
\]
i.e.\ $i\sim \sigma^{-1}(i)$.

Therefore, for every $\sigma$ there exists a set $I_\sigma\subseteq[d]$ with
$|I_\sigma|\ge d-(t-1)=d-t+1$ such that $i\sim \sigma^{-1}(i)$ for all $i\in I_\sigma$.
Since $\sigma$ ranges over all permutations, so does $\sigma^{-1}$, hence:
for every $\tau\in\mathfrak S_d$ there exists $J_\tau\subseteq[d]$ with $|J_\tau|=d-t+1$ such that
$i\sim \tau(i)$ for all $i\in J_\tau$.

Now apply Lemma~\ref{lem:counting} with this equivalence relation and with parameter $t':=d-t+1$.
We obtain an equivalence class $S\subseteq[d]$ with
\[
|S|\ge \left\lceil \frac{d+t'}2\right\rceil
=\left\lceil \frac{2d-t+1}{2}\right\rceil
=d-\left\lfloor \frac{t-1}{2}\right\rfloor,
\]
and $u_i\equiv_m u_j$ for all $i,j\in S$.
This is exactly alternative (ii) in Theorem~\ref{thm:main1} (taking $I=S$).

\end{proof}

\appendix
\section{Proof of Lemma~\ref{lem:sum}(ii)--(iii)}\label{app:sumproof}

\medskip
\noindent (ii)
Let
\[
\begin{aligned}
A_0&=\{a\in A:\ p_1(a)q_1(a)\neq 0\},\\
B_0&=\{b\in B:\ p_2(b)q_2(b)\neq 0\}.
\end{aligned}
\]
Each polynomial has at most $\delta$ real zeros, so $|A_0|\ge n-2\delta$ and $|B_0|\ge n-2\delta$.
(If $n\le 10\delta$ the desired bound is trivial after adjusting the implicit constant,
so we may assume $|A_0|\sim_\delta n$ and $|B_0|\sim_\delta n$.)

Apply Theorem~\ref{thm:enr} with $f=g=\log\abs{\cdot}$, and set
\[
\begin{aligned}
S&=\{(\log\abs{p_1(a)},\,\log\abs{q_1(a)}):\ a\in A_0\},\\
T&=\LogAbsset(p_2(B_0))\times \LogAbsset(q_2(B_0)).
\end{aligned}
\]

We claim the curve $\{(\log\abs{p_1(t)},\log\abs{q_1(t)})\}$ is not contained in an affine line
under $p_1\not\equiv_m q_1$.
Indeed, if it were contained in a line, then there exist $\alpha,\beta,\gamma\in\mathbb R$ with
$(\alpha,\beta)\neq(0,0)$ such that
\[
\alpha\log\abs{p_1(t)}+\beta\log\abs{q_1(t)}=\gamma
\]
for all $t$ in the domain.
If $\alpha=0$, then $\log\abs{q_1(t)}$ is constant on an infinite set, forcing $q_1$ to be constant,
a contradiction. Similarly, if $\beta=0$, then $p_1$ is constant, a contradiction. Hence $\alpha\beta\neq 0$.

Letting $t\to+\infty$ and using
$\log\abs{p_1(t)}=(\deg p_1)\log t+O(1)$ and
$\log\abs{q_1(t)}=(\deg q_1)\log t+O(1)$ gives
$\alpha\deg(p_1)+\beta\deg(q_1)=0$, hence $\alpha/\beta\in\mathbb Q$.
Scaling, we may assume $\alpha=m$ and $\beta=-n$ with coprime integers $m,n\ge 1$, so
$m\log\abs{p_1(t)}-n\log\abs{q_1(t)}=\gamma$, hence
$\abs{p_1(t)}^m=e^\gamma \abs{q_1(t)}^n$.
Squaring yields $p_1(t)^{2m}=e^{2\gamma}q_1(t)^{2n}$.
Since $p_1,q_1$ are monic, we must have $e^{2\gamma}=1$, hence $p_1^{2m}=q_1^{2n}$,
i.e.\ $p_1\equiv_m q_1$, a contradiction.

Next, by the fiber bound applied to $\abs{p_2}$ and $\abs{q_2}$,
\[
|\LogAbsset(p_2(B_0))|
=|\Absset(p_2(B_0))|
\gg_\delta n,
\qquad
|\LogAbsset(q_2(B_0))|
=|\Absset(q_2(B_0))|
\gg_\delta n,
\]
hence $|T|\gg_\delta n^2$.
Theorem~\ref{thm:enr} gives $|S+T|\gg_\delta n^{5/2}$.
Since
\[
\begin{aligned}
S+T\subset {}&
\bigl(\LogAbsset(p_1(A_0))+\LogAbsset(p_2(B_0))\bigr)\\
&\times
\bigl(\LogAbsset(q_1(A_0))+\LogAbsset(q_2(B_0))\bigr),
\end{aligned}
\]
we obtain
\[
\begin{aligned}
&\bigl|\LogAbsset(p_1(A_0))+\LogAbsset(p_2(B_0))\bigr|\cdot
 \bigl|\LogAbsset(q_1(A_0))+\LogAbsset(q_2(B_0))\bigr|\\
&\ge |S+T|\gg_\delta n^{5/2}.
\end{aligned}
\]

Because $\log$ is injective on $(0,\infty)$ and $\log(xy)=\log x+\log y$, we have the set identity
\[
\LogAbsset(U)+\LogAbsset(V)
=\Logset\bigl(\Absset(U)\cdot \Absset(V)\bigr)
\qquad(U,V\subset\mathbb R\setminus\{0\}),
\]
and therefore
\[
\bigl|\LogAbsset(U)+\LogAbsset(V)\bigr|
=\bigl|\Absset(U)\cdot \Absset(V)\bigr|.
\]
Applying this with $U=p_1(A_0)$ and $V=p_2(B_0)$ (and similarly for $q_1,q_2$) yields
\[
\begin{aligned}
&\bigl|\Absset(p_1(A_0))\cdot \Absset(p_2(B_0))\bigr|\cdot
\bigl|\Absset(q_1(A_0))\cdot \Absset(q_2(B_0))\bigr|
\gg_\delta n^{5/2}.
\end{aligned}
\]

Finally, note the exact identity
\[
\Absset(U)\cdot \Absset(V)=\Absset(U\cdot V)\qquad (U,V\subset\mathbb R),
\]
since $|u||v|=|uv|$.
Applying this with $U=p_1(A_0)$ and $V=p_2(B_0)$ gives
\[
\bigl|\Absset(p_1(A_0))\cdot \Absset(p_2(B_0))\bigr|
=\bigl|\Absset\bigl(p_1(A_0)\cdot p_2(B_0)\bigr)\bigr|
\le |p_1(A)\cdot p_2(B)|.
\]
The analogous bound holds for $q_1,q_2$, completing the proof of (ii).

\medskip
\noindent (iii)
Let
\[
A_1=\{a\in A:\ q_1(a)\neq 0\},\qquad
B_1=\{b\in B:\ q_2(b)\neq 0\},
\]
so $|A_1|\ge n-\delta$ and $|B_1|\ge n-\delta$.

Apply Theorem~\ref{thm:enr} with $f=\mathrm{id}$ and $g=\log\abs{\cdot}$, and set
\[
\begin{aligned}
S&=\{(p_1(a),\,\log\abs{q_1(a)}):\ a\in A_1\},\\
T&=p_2(B)\times \LogAbsset(q_2(B_1)).
\end{aligned}
\]
We check that the curve $\{(p_1(t),\log\abs{q_1(t)})\}$ is not contained in an affine line:
suppose for contradiction that there exist $\alpha,\beta,\gamma\in\mathbb R$ with $(\alpha,\beta)\neq(0,0)$ such that
\[
\alpha p_1(t)+\beta\log\abs{q_1(t)}=\gamma
\]
for all $t$ in the domain.
If $\beta=0$, then $p_1$ is constant, a contradiction. Thus $\beta\neq 0$.
Then $\alpha p_1(t)=\gamma-\beta\log\abs{q_1(t)}=O(\log t)$ as $t\to+\infty$.
Since $p_1$ is nonconstant, $p_1(t)$ grows polynomially in $t$, so this forces $\alpha=0$.
But then $\log\abs{q_1(t)}$ is constant on an infinite set, forcing $q_1$ to be constant, a contradiction.
Thus Theorem~\ref{thm:enr} applies.

Also $|p_2(B)|\gg_\delta n$ and $|\LogAbsset(q_2(B_1))|\gg_\delta n$, hence $|T|\gg_\delta n^2$,
so $|S+T|\gg_\delta n^{5/2}$.
Finally, using $\Absset(U)\cdot \Absset(V)=\Absset(U\cdot V)$, we have
\[
\bigl|\LogAbsset(q_1(A_1))+\LogAbsset(q_2(B_1))\bigr|
=\bigl|\Absset(q_1(A_1))\cdot \Absset(q_2(B_1))\bigr|
=\bigl|\Absset\bigl(q_1(A_1)\cdot q_2(B_1)\bigr)\bigr|
\le |q_1(A)\cdot q_2(B)|.
\]
This proves (iii).

\section{Proof of Theorem~\ref{thm:jrt1}: multiplicative and mixed cases}\label{app:jrt-cases}
\medskip
\noindent{ \emph{The multiplicative--multiplicative case.}}
Assume
\[
P(x,y)=f\big(u_1(x)u_2(y)\big),\qquad Q(x,y)=g\big(v_1(x)v_2(y)\big).
\]
By dividing $u_1,u_2$ by their leading coefficients and absorbing the resulting nonzero constant into $f$
(and similarly for $v_1,v_2,g$), we may assume $u_1,u_2,v_1,v_2$ are monic.

If $u_1\equiv_m v_1$ and $u_2\equiv_m v_2$, then (ii) holds and we are done.
Otherwise, after possibly swapping the roles of $A,B$ and swapping $(u_1,v_1)$ with $(u_2,v_2)$,
we may assume $u_1\not\equiv_m v_1$.
Applying Lemma~\ref{lem:sum}(ii) gives
\[
|u_1(A)\cdot u_2(B)|\cdot |v_1(A)\cdot v_2(B)|\gg_{\delta} n^{5/2}.
\]
By Lemma~\ref{lem:size} applied to $f$ and $g$ on these two product sets, we have
\[
|P(A,B)|=|f(u_1(A)\cdot u_2(B))|\gg_{\delta} |u_1(A)\cdot u_2(B)|,
\]
\[
|Q(A,B)|=|g(v_1(A)\cdot v_2(B))|\gg_{\delta} |v_1(A)\cdot v_2(B)|.
\]
Therefore $|P(A,B)|\cdot |Q(A,B)|\gg_{\delta} n^{5/2}$, hence again
\[
\max\{|P(A,B)|,|Q(A,B)|\}\gg_{\delta} n^{5/4}.
\]

\medskip
\noindent{\emph{The mixed additive--multiplicative case.}}
Without loss of generality, assume
\[
P(x,y)=f\big(u_1(x)+u_2(y)\big),\qquad Q(x,y)=g\big(v_1(x)v_2(y)\big).
\]
Applying Lemma~\ref{lem:sum}(iii) (with $p_1=u_1$, $p_2=u_2$, $q_1=v_1$, $q_2=v_2$) gives
\[
|u_1(A)+u_2(B)|\cdot |v_1(A)\cdot v_2(B)|\gg_{\delta} n^{5/2}.
\]
Using Lemma~\ref{lem:size} as before,
\[
|P(A,B)|\gg_{\delta} |u_1(A)+u_2(B)|,\qquad |Q(A,B)|\gg_{\delta} |v_1(A)\cdot v_2(B)|.
\]
Thus $|P(A,B)|\cdot |Q(A,B)|\gg_{\delta} n^{5/2}$, and hence
\[
\max\{|P(A,B)|,|Q(A,B)|\}\gg_{\delta} n^{5/4}.
\]

\medskip
Combining the three cases, we conclude that
$\max\{|P(A,B)|,|Q(A,B)|\}\gg_{\delta} n^{5/4}$
unless (i) or (ii) holds.

\section{The complete proof of Theorem \ref{thm:main2}}\label{app:mm}

\smallskip
\noindent{ \emph{The multiplicative--multiplicative case.}}
Assume we are in the multiplicative--multiplicative case:
\[
P(x_{1},\ldots,x_{d})=f\!\Big(\prod_{i=1}^{d}u_{i}(x_{i})\Big),
\qquad
Q(x_{1},\ldots,x_{d})=g\!\Big(\prod_{i=1}^{d}v_{i}(x_{i})\Big),
\]
with $f,g,u_{1},\ldots,u_{d},v_{1},\ldots,v_{d}\in\mathbb R[x]$ nonconstant.

\medskip
\noindent\emph{Monic normalization (needed for the $\equiv_m$-vs-line argument).}
Write $u_i(x)=c_i\,\widetilde u_i(x)$ and $v_i(x)=d_i\,\widetilde v_i(x)$ where $c_i,d_i\in\mathbb R^\times$ and
$\widetilde u_i,\widetilde v_i$ are monic.
Define new outer polynomials
\[
\widetilde f(z):=f\Big(\big(\prod_{i=1}^d c_i\big)\,z\Big),
\qquad
\widetilde g(z):=g\Big(\big(\prod_{i=1}^d d_i\big)\,z\Big).
\]
Then
\[
P=\widetilde f\!\Big(\prod_{i=1}^d \widetilde u_i(x_i)\Big),\qquad
Q=\widetilde g\!\Big(\prod_{i=1}^d \widetilde v_i(x_i)\Big),
\]
and for every choice of sets $A_1,\dots,A_d$ we have
$P(A_1,\dots,A_d)=\widetilde f(\prod_i \widetilde u_i(A_i))$ and similarly for $Q$.
Thus we may (and do) assume from now on that every $u_i$ and $v_i$ is monic.

\medskip
For each $i$, since $\deg u_i,\deg v_i\le \delta$, the set
\[
Z_i:=\{a\in A_i:\ u_i(a)=0 \ \text{or}\ v_i(a)=0\}
\]
has size $|Z_i|\ll_\delta 1$. Replace $A_i$ by $A_i':=A_i\setminus Z_i$ and write
$n':=\min_i |A_i'|$, so $n'\sim_\delta n$. On $A_i'$ the quantities $\log|u_i(\cdot)|$ and $\log|v_i(\cdot)|$ are well-defined.

Define the log-sumsets
\[
\widetilde U_d:=\sum_{i=1}^{d}\log|u_i(A_i')|,
\qquad
\widetilde V_d:=\sum_{i=1}^{d}\log|v_i(A_i')|.
\]

We prove the following.

\smallskip
\noindent{\bf Statement $\widetilde{\mathcal P}(d,t)$.}
If there exists $I\subseteq[d]$ with $|I|=t$ such that $u_i\not\equiv_m v_i$ for all $i\in I$, then
\[
|\widetilde U_d|\cdot |\widetilde V_d|\gg_{\delta} (n')^{3-\frac1{2^{t}}}.
\]

\smallskip
\noindent{\it Base case $\widetilde{\mathcal P}(2,1)$.}
Assume wlog that $u_1\not\equiv_m v_1$. Set
\[
S:=\{(\log|u_1(a)|,\,\log|v_1(a)|):\ a\in A_1'\}\subset\mathbb R^2,
\qquad
T:=\log|u_2(A_2')|\times \log|v_2(A_2')|.
\]
Then $S+T\subseteq \widetilde U_2\times \widetilde V_2$, hence
\[
|\widetilde U_2|\cdot |\widetilde V_2|\ge |S+T|.
\]

We claim the curve
\[
\{(\log|u_1(x)|,\log|v_1(x)|): x\in\mathbb R,\ u_1(x)v_1(x)\neq 0\}
\]
is not contained in an affine line.
Indeed, if it were, there exist $\alpha,\beta,\gamma\in\mathbb R$ with $(\alpha,\beta)\neq(0,0)$ such that
\[
\alpha\log|u_1(x)|+\beta\log|v_1(x)|=\gamma
\]
for all $x$ in the domain.
If $\alpha=0$, then $\log|v_1(x)|$ is constant on an infinite set, forcing $v_1$ to be constant, a contradiction.
Similarly, if $\beta=0$, then $u_1$ is constant, a contradiction. Hence $\alpha\beta\neq 0$.

Letting $x\to+\infty$ and using
$\log|u_1(x)|=(\deg u_1)\log x+O(1)$ and
$\log|v_1(x)|=(\deg v_1)\log x+O(1)$ gives
$\alpha\deg(u_1)+\beta\deg(v_1)=0$, hence $\alpha/\beta\in\mathbb Q$.
Scaling, we may assume $\alpha=m$ and $\beta=-n$ for coprime integers $m,n\ge 1$, so
$m\log|u_1(x)|-n\log|v_1(x)|=\gamma$, i.e.
$|u_1(x)|^m=e^{\gamma}|v_1(x)|^n$.
Squaring yields $u_1(x)^{2m}=e^{2\gamma}v_1(x)^{2n}$ as polynomials.
Since $u_1,v_1$ are monic, comparing leading coefficients gives $e^{2\gamma}=1$, hence $u_1^{2m}=v_1^{2n}$ and so
$u_1\equiv_m v_1$, contradicting the assumption.
This proves the claim.

Thus Theorem~\ref{thm:enr} applied with $f=g=\log|\cdot|$, $p=u_1$, $q=v_1$, and $A=A_1'$ gives
\[
|S+T|\gg_\delta \min\{\,|A_1'||T|,\ |A_1'|^{3/2}|T|^{1/2}\,\}.
\]
Since $\deg u_2,\deg v_2\le\delta$ and $u_2,v_2$ are nonconstant, we have
$|u_2(A_2')|\ge |A_2'|/\delta$ and $|v_2(A_2')|\ge |A_2'|/\delta$.
Moreover, the map $y\mapsto \log|y|$ has fibers of size at most $2$ on $\mathbb R\setminus\{0\}$, hence
\[
|\log|u_2(A_2')||\ge \frac12|u_2(A_2')|\gg_\delta n',
\qquad
|\log|v_2(A_2')||\ge \frac12|v_2(A_2')|\gg_\delta n'.
\]
Therefore $|T|\gg_\delta (n')^2$, so in particular $|T|\ge |A_1'|$ for $n$ large (and the small-$n$ case is absorbed
by adjusting implicit constants). Hence the minimum is the second term and we obtain
\[
|S+T|\gg_\delta |A_1'|^{3/2}|T|^{1/2}\gg_\delta (n')^{3/2}\cdot((n')^2)^{1/2}=(n')^{5/2},
\]
which proves $\widetilde{\mathcal P}(2,1)$.

\smallskip
\noindent{\it Induction.}
Assume $\widetilde{\mathcal P}(d,t)$. We show it implies both
$\widetilde{\mathcal P}(d+1,t+1)$ and $\widetilde{\mathcal P}(d+1,t)$.

\smallskip
\noindent{\bf (1) $\widetilde{\mathcal P}(d,t)\Rightarrow \widetilde{\mathcal P}(d+1,t+1)$.}
Let $I\subseteq[d+1]$ with $|I|=t+1$ and $u_i\not\equiv_m v_i$ for all $i\in I$.
Relabel so that $1\in I$ and set
\[
S:=\{(\log|u_1(a)|,\,\log|v_1(a)|):\ a\in A_1'\},\qquad
T:=\widetilde U'\times \widetilde V',
\]
where
\[
\widetilde U':=\sum_{i=2}^{d+1}\log|u_i(A_i')|,
\qquad
\widetilde V':=\sum_{i=2}^{d+1}\log|v_i(A_i')|.
\]
Then $S+T\subseteq \widetilde U_{d+1}\times\widetilde V_{d+1}$, so
$|\widetilde U_{d+1}|\cdot|\widetilde V_{d+1}|\ge |S+T|$.
As above, $u_1\not\equiv_m v_1$ implies the curve of $(\log|u_1|,\log|v_1|)$ is not contained in an affine line, hence by Theorem~\ref{thm:enr},
\[
|S+T|\gg_\delta \min\{\,|A_1'||T|,\ |A_1'|^{3/2}|T|^{1/2}\,\}.
\]
By $\widetilde{\mathcal P}(d,t)$ applied to the remaining $d$ indices (since $I\setminus\{1\}$ has size $t$),
\[
|T|=|\widetilde U'|\cdot|\widetilde V'|\gg_\delta (n')^{3-\frac1{2^t}}.
\]
In particular $|T|\ge |A_1'|$ for $n$ large, so the minimum is the second term and
\[
|S+T|\gg_\delta (n')^{3/2}\cdot \Big((n')^{3-\frac1{2^t}}\Big)^{1/2}
=(n')^{3-\frac1{2^{t+1}}},
\]
which is $\widetilde{\mathcal P}(d+1,t+1)$.

\smallskip
\noindent{\bf (2) $\widetilde{\mathcal P}(d,t)\Rightarrow \widetilde{\mathcal P}(d+1,t)$.}
If $I\subseteq[d+1]$ with $|I|=t$ and $d+1\notin I$, then applying $\widetilde{\mathcal P}(d,t)$ to indices $1,\dots,d$ gives
$|\widetilde U_d|\cdot|\widetilde V_d|\gg_\delta (n')^{3-\frac1{2^t}}$.
Since adding a nonempty set cannot decrease cardinality, we have
$|\widetilde U_{d+1}|\ge|\widetilde U_d|$ and $|\widetilde V_{d+1}|\ge|\widetilde V_d|$,
hence $\widetilde{\mathcal P}(d+1,t)$ follows (the case $d+1\in I$ is the same after relabeling).

\smallskip
\noindent Therefore $\widetilde{\mathcal P}(d,t)$ holds for all $d\ge2$ and $1\le t\le d-1$.

\smallskip
\noindent Returning to our situation, since $|\{i:\ u_i\not\equiv_m v_i\}|\ge t$, we obtain
\[
|\widetilde U_d|\cdot |\widetilde V_d|\gg_\delta (n')^{3-\frac1{2^t}}\sim_\delta n^{3-\frac1{2^t}}.
\]

Now set
\[
W_P:=\prod_{i=1}^d u_i(A_i'),\qquad W_Q:=\prod_{i=1}^d v_i(A_i').
\]
Since all factors are nonzero on $A_i'$, every element of $W_P$ and $W_Q$ is nonzero.
Moreover, by $\log|xy|=\log|x|+\log|y|$ we have the exact identities of sets
\[
\log|W_P|:=\{\log|w|:\ w\in W_P\}=\sum_{i=1}^d \log|u_i(A_i')|=\widetilde U_d,
\]
\[
\log|W_Q|:=\{\log|w|:\ w\in W_Q\}=\sum_{i=1}^d \log|v_i(A_i')|=\widetilde V_d.
\]
Therefore $|W_P|\ge |\widetilde U_d|$ and $|W_Q|\ge |\widetilde V_d|$.

Since $\deg f,\deg g\le \delta$, every value of $f$ (resp.\ $g$) has at most $\delta$ preimages, hence
\[
|P(A_1',\ldots,A_d')|=|f(W_P)|\ge \frac{|W_P|}{\delta}\gg_\delta |\widetilde U_d|,
\qquad
|Q(A_1',\ldots,A_d')|=|g(W_Q)|\ge \frac{|W_Q|}{\delta}\gg_\delta |\widetilde V_d|.
\]
Therefore
\[
|P(A_1,\ldots,A_d)|\cdot |Q(A_1,\ldots,A_d)|
\gg_\delta
|\widetilde U_d|\cdot|\widetilde V_d|
\gg_\delta
n^{3-\frac1{2^t}},
\]
and the same contradiction as in the additive--additive case follows.

\medskip
\noindent{\emph{The mixed additive--multiplicative case.}}
Assume $P$ is additive and $Q$ is multiplicative:
\[
P(x_{1},\ldots,x_{d})=f\!\Big(\sum_{i=1}^{d}u_{i}(x_{i})\Big),
\qquad
Q(x_{1},\ldots,x_{d})=g\!\Big(\prod_{i=1}^{d}v_{i}(x_{i})\Big),
\]
with $f,g,u_i,v_i$ nonconstant.
As above, remove zeros of the multiplicative factors: let
$A_i':=\{a\in A_i:\ v_i(a)\ne 0\}$, so $|A_i\setminus A_i'|\ll_\delta 1$ and
$n':=\min_i |A_i'|\sim_\delta n$.
Define
\[
U_d:=\sum_{i=1}^d u_i(A_i'),\qquad
\widetilde V_d:=\sum_{i=1}^d \log|v_i(A_i')|.
\]

We claim that for every $d\ge2$,
\[
|U_d|\cdot|\widetilde V_d|\gg_\delta (n')^{3-\frac1{2^{d-1}}}.
\tag{$\star_d$}
\]
\emph{Proof by induction on $d$.}
For $d=2$, set
\[
S:=\{(u_1(a),\,\log|v_1(a)|):\ a\in A_1'\}\subset\mathbb R^2,
\qquad
T:=u_2(A_2')\times \log|v_2(A_2')|.
\]
Then $S+T\subseteq U_2\times \widetilde V_2$, hence $|U_2|\cdot|\widetilde V_2|\ge |S+T|$.
The curve $\{(u_1(x),\log|v_1(x)|):x\in\mathbb R,\ v_1(x)\ne0\}$ is not contained in an affine line:
suppose for contradiction that there exist $\alpha,\beta,\gamma\in\mathbb R$ with $(\alpha,\beta)\neq(0,0)$ such that
\[
\alpha u_1(x)+\beta\log|v_1(x)|=\gamma
\]
for all $x$ in the domain.
If $\beta=0$, then $u_1$ is constant, a contradiction. Thus $\beta\neq 0$.
Then
\[
\alpha u_1(x)=\gamma-\beta\log|v_1(x)|=O(\log x)\qquad(x\to+\infty).
\]
Since $u_1$ is nonconstant, it has polynomial growth, so this forces $\alpha=0$.
But then $\log|v_1(x)|$ is constant on an infinite set, forcing $v_1$ to be constant, a contradiction.
Thus Theorem~\ref{thm:enr} with $(f,g)=(\mathrm{id},\log|\cdot|)$ applies and gives
\[
|S+T|\gg_\delta \min\{\,|A_1'||T|,\ |A_1'|^{3/2}|T|^{1/2}\,\}.
\]
As before, $|u_2(A_2')|\gg_\delta n'$ and $|\log|v_2(A_2')||\gg_\delta n'$, hence $|T|\gg_\delta (n')^2$,
so the minimum is the second term and we obtain $|S+T|\gg_\delta (n')^{5/2}$, which is $(\star_2)$.

Now assume $(\star_d)$ holds. For $d+1$, define
\[
S:=\{(u_1(a),\,\log|v_1(a)|):\ a\in A_1'\},\qquad
T:=U'\times \widetilde V',
\]
where $U':=\sum_{i=2}^{d+1}u_i(A_i')$ and
$\widetilde V':=\sum_{i=2}^{d+1}\log|v_i(A_i')|$.
Then $S+T\subseteq U_{d+1}\times \widetilde V_{d+1}$, so
$|U_{d+1}|\cdot|\widetilde V_{d+1}|\ge |S+T|$.
Apply Theorem~\ref{thm:enr} again, 
\[
|S+T|\gg_\delta \min\{\,|A_1'||T|,\ |A_1'|^{3/2}|T|^{1/2}\,\}.
\]
By the induction hypothesis $(\star_d)$ applied to indices $2,\dots,d+1$,
\[
|T|=|U'|\cdot|\widetilde V'|\gg_\delta (n')^{3-\frac1{2^{d-1}}}\ge |A_1'|
\]
for $n$ large, so the minimum is the second term and we get
\[
|U_{d+1}|\cdot|\widetilde V_{d+1}|
\gg_\delta
(n')^{3/2}\cdot \Big((n')^{3-\frac1{2^{d-1}}}\Big)^{1/2}
=
(n')^{3-\frac1{2^{d}}},
\]
which is $(\star_{d+1})$. This proves $(\star_d)$ for all $d\ge2$.

\smallskip
\noindent Finally, since $\deg f,\deg g\le\delta$, we have
\[
|P(A_1',\ldots,A_d')|=|f(U_d)|\ge |U_d|/\delta\gg_\delta |U_d|.
\]
Also, letting $W_Q:=\prod_{i=1}^d v_i(A_i')$, we have the exact identity
\[
\log|W_Q|=\sum_{i=1}^d \log|v_i(A_i')|=\widetilde V_d,
\]
hence $|W_Q|\ge |\widetilde V_d|$.
Therefore
\[
|Q(A_1',\ldots,A_d')|=|g(W_Q)|\ge |W_Q|/\delta \gg_\delta |\widetilde V_d|.
\]
Consequently,
\[
|P(A_1,\ldots,A_d)|\cdot |Q(A_1,\ldots,A_d)|
\gg_\delta
|U_d|\cdot|\widetilde V_d|
\gg_\delta
n^{3-\frac1{2^{d-1}}},
\]
and so
\[
\max\{|P(A_1,\ldots,A_d)|,\ |Q(A_1,\ldots,A_d)|\}
\ge (|P|\cdot |Q|)^{1/2}
\gg_\delta
n^{\frac32-\frac1{2^{d}}}.
\]
Since $t+1\le d$, we have $\frac1{2^{d}}\le \frac1{2^{t+1}}$, hence this bound is stronger than
$n^{\frac32-\frac1{2^{t+1}}}$, giving the desired contradiction in the mixed case as well.

\bibliographystyle{amsalpha}
\bibliography{ref}
\end{document}